\documentclass[article,12pt]{amsart}

\usepackage[a4paper,hmargin=3.5cm,vmargin=4cm,rmargin=3cm,lmargin=3cm]{geometry}
\usepackage{amsmath, amsthm,  amsfonts,amssymb,amscd,amstext}
\usepackage{enumerate}
\usepackage{graphicx, color}
\usepackage[x11names]{xcolor}
\usepackage{hyperref}

\usepackage{mathrsfs}

\usepackage{fancyhdr}
\newtheorem{theorem}{Theorem}[section]

\newtheorem{corollary}[theorem]{Corollary}
\newtheorem{lemma}[theorem]{Lemma}
\newtheorem{example}[theorem]{Example}

\theoremstyle{definition}
\newtheorem{definition}[theorem]{Definition}

\newtheorem{remark}[theorem]{Remark}

\numberwithin{equation}{section}
\numberwithin{figure}{section}

\newcommand\Hcal{\mathcal{H}}

\def\Rcal{\mathcal{R}}
\newcommand\rcal{\mathcal{R}}

\newcommand\Fcal{\mathcal{F}}

\def\Ascr{\mathscr{A}}

\newcommand\Cscr{\mathscr{C}}

\newcommand\Oscr{\mathscr{O}}

\newcommand\C{\mathbb{C}}
\newcommand\D{\overline{\mathbb D}}

\renewcommand\D{\mathbb D}

\newcommand\K{\mathbb{K}}

\newcommand\N{\mathbb{N}}
\renewcommand\P{\mathbb{P}}
\newcommand\R{\mathbb{R}}

\newcommand\Z{\mathbb{Z}}

\def\c{\mathbb{C}}

\def\n{\mathbb{N}}
\renewcommand\r{\mathbb{R}}

\def\z{\mathbb{Z}}
\renewcommand\k{\mathbb{K}}

\newcommand\pgot{\mathfrak{p}}

\newcommand\Agot{\mathfrak{A}}

\def\Sgot{\mathfrak{S}}

\newcommand\wt{\widetilde}
\newcommand\wh{\widehat}

\def\dist{\rm{dist}}
\def\length{\rm{length}}

\def\Flux{\mathrm{Flux}}

\begin{document}
	
\fancyhead[LO]{Carleman approximation by minimal surfaces and directed holomorphic curves} 
\fancyhead[RE]{I.\ Castro-Infantes and B.\ Chenoweth}
\fancyhead[RO,LE]{\thepage}

\thispagestyle{empty}

\begin{center}
	{\LARGE \begin{center}
			\textbf{Carleman approximation by conformal minimal immersions and directed holomorphic curves}
		\end{center}}
	
	\vspace*{3mm}
	
	{\large\bf  Ildefonso Castro-Infantes \; and \; Brett Chenoweth}
\end{center}

\begin{quote}
	{\small
	\noindent {\bf Abstract} 	}
		Let $\Rcal$ be an open Riemann surface. 
		In this paper we prove that every continuous function \( M \to \R^n \), $n\ge 3$, defined on a divergent Jordan arc \( M
		\subset 
		\Rcal \) can be approximated in the Carleman sense by conformal minimal immersions; thus providing a new generalization of Carleman's theorem.
		In fact, we prove that this result remains true for null curves and many other classes of directed holomorphic immersions for which the directing variety satisfies a certain flexibility property.
		Furthermore, the constructed immersions may be chosen to be complete or proper under natural assumptions on the variety and the continuous map.
		
		As a consequence we give an approximate solution to a Plateau problem for divergent Jordan curves in the Euclidean spaces.
		
		
\end{quote}

\section{Introduction}
Carleman's theorem \cite{carleman1927} is a result from the early 20th century that asserts: given a complex valued function \( f\colon \r\to\c \) and a strictly positive function \(\epsilon\colon\r\to\r_+\), there exists an entire function \( g \colon\c\to\c \) such that
\( 
|f(x)-g(x) | < \epsilon(x)
\)
for every $x\in \r$.
Observe that the conclusion of this theorem is better than uniform approximation on $\R$ by entire functions.
In this paper we are interested in whether there are analogues of Carleman's theorem in the context of minimal surface theory; a classical and important area of geometry research.
This is a timely question because  recent  connections between complex analysis and minimal surface theory have seen classical approximation results, in particular those of Runge and Mergelyan, proven in the context of minimal surface theory, see \cite{AlarconLopez2012JDG,AlarconForstneric2014IM}.
In this paper we prove the following.

\begin{theorem}[Carleman theorem for conformal minimal immersions]\label{th:simpleminimal}

	Let \( S \) be a smooth properly embedded image of $\r$ in an open Riemann surface \( \Rcal \). 
	Given a continuous map $X\colon S\to\r^n$, $n\ge 3$, 
	and a positive function $\epsilon\colon S\to\r_+$, there exists a complete conformal minimal immersion $\wt X\colon \Rcal\to \r^n$ such that
	\[
	\| \wt X(p)- X(p) \| < \epsilon(p), \quad p\in S.
	\]
	If in addition \( n \geq 5 \), then we may ensure that \( \wt X \) is also injective.
\end{theorem}
\noindent This result is the first Carleman type theorem for conformal minimal immersions in  $\R^n$.
%
%
In fact we prove that the previous result remains true not only for minimal surfaces, but also for directed immersions: a more general family of holomorphic immersions which includes {\em null curves}. A null curve is a holomorphic immersion $F\colon \Rcal\to\c^n$ directed by the null quadric
\begin{equation}\label{eq:pnullquadric}
	\Agot:=\{ z\in\c^n : z_1^2+\ldots+z_n^2=0\}.
\end{equation}
This family of holomorphic immersions is closely related to conformal minimal surfaces in the sense that the real and imaginary part of a null curve is a conformal minimal immersion. Conversely, any conformal minimal immersions defined over a simply connected domain is the real (or imaginary) part of a null curve. Recent techniques coming from complex analysis have been used in the study of minimal surfaces, see \textsection \ref{connections} or the survey \cite{AlarconForstneric2017Survey} for more details.

\begin{theorem}[Carleman theorem for directed immersions]\label{th:easyCarleman}
	Let \( \Rcal \) and $S\subset \Rcal$ be as in Theorem \ref{th:simpleminimal}.
	Let \( \Sgot \) be an irreducible closed conical subvariety in \( \C^n \), \( n \geq 3\), which is contained in no hyperplane and such that \( \Sgot \setminus \{ 0 \}  \) is an Oka manifold.
	For every \( \Sgot\)-immersion \( F\colon S \to \C^n \) (see Def.\,\ref{def:Simmersion}) and every positive continuous function \( \epsilon \colon S\to\r_+\), there is an {injective} \( \Sgot\)-immersion \( \wt F\colon \Rcal \to \C^n \) such that
	\[
	\| \wt F(p) - F(p) \|< \epsilon(p),\quad \text{ for every } p \in S.
	\]
\end{theorem}
\noindent Strictly speaking, by applying \cite[Lemma 3.3]{AlarconCastro2018} infinitely often, it is enough to assume that $F$ is a continuous map. 
However, adding the assumption that $F$ is $\Cscr^1$ means our notion of $\Sgot$-immersion is consistent with the notion introduced in \cite{AlarconForstneric2014IM,AlarconForstnericLopez2016MZ}.

Theorem \ref{th:easyCarleman} is  a consequence of the more general  Theorem \ref{th:Carleman} where the approximating set $S$ is a Carleman admissible subset, see Definition \ref{def:Cadmissible}, and the approximating map  interpolates to a finite order along a discrete subset of $\Rcal$.
Carleman admissible sets naturally generalize the admissible sets on which Mergelyan results for conformal minimal immersions and directed holomorphic curves have been proven.
An arbitrary open Riemann surface has many Carleman admissible sets $S$ for which $\mathring S$ has infinitely many connected components and  is not relatively compact.
Such  sets are not covered by previous theorems but are useful to consider because they allow us to construct  interesting examples of minimal surfaces, see  Corollary \ref{cor:oneminimalsurfacetorulethemall}.

Furthermore, we may prove global properties for the solutions that we construct.
The next result shows that the approximation may be done by complete $\Sgot$-immersions under natural assumptions on $\Sgot$.
\begin{theorem}\label{th:completeness}
	The $\Sgot$-immersion $\wt F\colon \Rcal\to\c^n$, $n\ge 3$, constructed in  Theorem \ref{th:easyCarleman} may be chosen to be complete provided that
	$\Sgot\cap\{z_1=1\}$ is an Oka manifold and the coordinate projection $\pi_1\colon \Sgot\to\c$ onto the $z_1$-axis admits a local holomorphic section $h$ near $z=0\in\c$ with $h(0)\neq 0$.
\end{theorem}
\noindent On the other hand, for proper immersions we have the following result.
\begin{theorem}\label{th:properness}
	The $\Sgot$-immersion $\wt F\colon \Rcal\to\c^n$, $n\ge 3$, constructed in Theorem \ref{th:easyCarleman} may be chosen to be proper provided that
	the restricted map $F|_{S}$ is proper and that
	$\Sgot\cap\{z_j=1\}$ is an Oka manifold and the coordinate projection $\pi_j\colon \Sgot\to\c$ onto the $z_j$-axis admits a local holomorphic section $h_j$ near $z=0\in\c$ with $h_j(0)\neq 0$ for all $j=1,\ldots,n$. 
\end{theorem}
A proper map is complete but a complete map is not necessarily proper. Therefore Theorem \ref{th:properness} does not imply Theorem \ref{th:completeness} since  the initial data in Theorem \ref{th:completeness} need not be proper.

The additional assumptions on \( \Sgot \) in Theorems \ref{th:completeness} and \ref{th:properness} are in particular satisfied by the null quadric $\Agot$ and hence Theorem \ref{th:easyCarleman} is applicable to {\em null curves}, that is, to holomorphic immersions directed by the null quadric. Hence, Theorem \ref{th:easyCarleman} applies to conformal minimal immersions $X$ with vanishing flux; see \textsection \ref{connections} for details.
Moreover, by modifying the proof slightly we may prove a more general result that ensures also control of the flux map, jet interpolation on a closed and discrete subset of $\Rcal$, properness of $\wt X$ if the image of $M$ by $X$ is proper, and injectivity of $\wt X$ if $n\ge 5$, see Theorem \ref{th:Carlemanminimal}.

The usual way, based on the ideas of Jorge and Xavier \cite{JorgeXavier1980AM} and Nadirashvili \cite{Nadirashvili1996IM}, to ensure completeness when constructing a conformal minimal immersion  is to approximate a concrete map defined on a labyrinth which has certain desirable properties. This method does not seem applicable here because $M$ impedes our ability to construct such a labyrinth. 
Therefore, we develop a new technique to prove completeness in Theorem \ref{th:completeness} that makes use of a different type of approximating set.
This new approach may  also be used to simplify the proofs of the previous results.

Several applications of our results are presented in \textsection \ref{applications}. These applications include an approximate solution to the Plateau problem for divergent paths or an example of a minimal surface approximately containing every conformal minimal surface \( \overline \D \to \R^n \) to any degree of accuracy, see Corollary \ref{cor:plat} and Corollary \ref{cor:oneminimalsurfacetorulethemall} respectively. 
We will also show, Corollary \ref{co:stratified}, that the basic case of our theorem remains true for (stratified) totally real sets.
In particular, combining our results with arguments from \cite{magnusson2016} we obtain the following.
\begin{enumerate}[\rm (1)]
	\item A totally real set \( M \subset \Rcal \) admits Carleman approximation by \( \Sgot \)-immersions if and only if \( M \) is holomorphically convex and has bounded E-hulls. 
	
	\item Suppose \( \Sgot \) satisfies the additional hypotheses of  Theorem \ref{th:easyCarleman}. A totally real set \( M \subset \Rcal \) admits Carleman approximation by complete \( \Sgot \)-immersions if and only if \( M \) is holomorphically convex and has bounded E-hulls. 
\end{enumerate}
Since the null quadric satisfies the additional hypotheses of Theorem \ref{th:easyCarleman}, see \cite[\textsection 2.3]{AlarconCastro2018}, we point out that item {\rm (2)} holds for null curves.

The main tools we need to prove our Carleman theorem are a  Mergelyan type approximation theorem and a method of gluing. 
Mergelyan's theorem and the tools needed to prove our gluing lemma, Lemma \ref{lem:gluing}, are provided by \cite{AlarconCastro2018,AlarconForstneric2014IM,AlarconForstnericLopez2016MZ}.
In particular, we find \cite[Lemma 3.3]{AlarconCastro2018} especially useful.
{We point out that} it would not suffice to simply glue the derivatives and then integrate because, amongst other problems, the integral is \emph{not} going to be a gluing in general.

Note that Carleman approximation by holomorphic functions was already  well understood;  approximating by holomorphic functions is clearly easier than approximating by \( \Sgot \)-immersions since the latter is a more constrained version of the former.
Call a subset \( E \) of an open Riemann surface \(\Rcal\) a Carleman approximation set if for every continuous function \( f \in \Cscr(E) \) which is holomorphic on the interior of \( E \) and every strictly positive continuous function \( \epsilon \in \Cscr(E) \), there exists \( g \in \Oscr(\Rcal)\)  such that 
\begin{equation}\label{carlemansense}
|f(x)-g(x) | < \epsilon(x), \quad \text{ for every \( x \in E \).}
\end{equation}
Nersesjan \cite{nersesjan1971} characterised Carleman approximation sets in the planar case in 1971, and later in 1986 Boivin \cite{boivin1986} obtained a characterisation for arbitrary open Riemann surfaces.
Furthermore, sets of Carleman approximation by harmonic functions have  been studied  and are well understood for connected open subsets of \( \R^n \); see  \cite{gardiner1995}.

Carleman's theorem has also been generalised to higher dimensions.
Manne, \O{}vrelid, and Wold \cite{manne2011} obtained a satisfying characterisation for a  totally real set \( M \) contained in a Stein manifold \( X \) to admit \( \Cscr^k \)-Carleman approximation, \( k \geq 1 \).
Magnusson and Wold \cite{magnusson2016} showed that if \( X = \C^n \) then the same characterisation holds more generally for statified totally real sets and for \( k \geq 0 \).
The second author \cite{Chenoweth2018Carleman}  has also proven Carleman results for maps from Stein manifolds into Oka manifolds.

The  reader may wonder the following. If \cite{Chenoweth2018Carleman}  gives us Carleman approximation of maps \( X \to Y \) where \( X \) is Stein and \( Y \) is Oka, and if \( \Sgot_* \) is an Oka manifold, then is it not the case that Theorem \ref{th:easyCarleman} (or at least the basic case) follows from this theorem?
The answer to this question seems to be no.
The naive idea would be to approximate the derivative and then integrate to get something approximating the original \( \Sgot \)-immersion.
However, assuming you were somehow able to fix the periods, one should not expect the error of the integral function to tend towards zero as \(  x \rightarrow \infty\) as the error might be cumulative. Hence this method would not give us Carleman approximation (however, it would give uniform approximation on \( S \)).

We conclude this introduction by framing some of our results in function theoretic terms.
Let \( M \subset \Rcal \) be a totally real set, see Definition \ref{def:totallyrealset}.
For continuous functions \( f\colon M \to \C^n\)  and \( \epsilon\colon M \to \R_+ \) define
\[
[f,\epsilon]=\{ g \in \Cscr(M,\C^n): \|f(x)-g(x)\|< \epsilon(x),\ \forall x\in M\}.
\]
Sets of this form form a basis for the strong  topology on \( \Cscr(M,\C^n)\).
Denote the collection of continuous maps \( M \to \C^n \) endowed with this topology by \( \Cscr_{S}(M,\C^n) \).
Let \( \text{PROP}(M,\C^n)\subset \Cscr_{S}(M,\C^n) \) be the set of all proper continuous maps  \( M \to \C^n \). One can prove that \( \text{PROP}(M,\C^n)\) is an open subset of \(  \Cscr_{S}(M,\C^n) \), see \cite[Theorem 1.5]{hirsch1976}.
Let \( \rho\colon\Cscr(\Rcal,\C^n)\to \Cscr_{S}(M,\C^n)  \) be the restriction operator.
By applying \cite[Lemma 3.3]{AlarconCastro2018} infinitely often one can prove that the space of \( \Sgot \)-immersions   is dense in \( \Cscr_S(M,\C^n)\).
\label{rem:sectioncontainingdense}
Our main result  thus  implies  the following:

\begin{enumerate}\item[a)] The set \( \rho(A) \) is dense in \(\Cscr_{S}(M,\C^n)\), where \( A \) is  the set of all \( \Sgot \)-immersions \( \Rcal \to \C^n \).
\end{enumerate}

\noindent For \( \Sgot \) satisfying the additional hypotheses of Theorem \ref{th:easyCarleman}:
\begin{enumerate}
	\item[b)] The set \( \rho(B) \) is dense in \(\Cscr_{S}(M,\C^n)\), where \( B \) is  the set of all complete \( \Sgot \)-immersions \( \Rcal \to \C^n \)
	\item[c)] The  set \( \rho(C)\)  is dense in  is dense in \(\text{PROP}(M,\C^n)\), where \( C \) is the set of all proper \( \Sgot \)-immersions \( \Rcal \to \C^n \).
\end{enumerate}

\section{Preliminaries}\label{sec:prelim}
Let $\n=\{1,2,3,\ldots\}$, $\z_+=\n \cup \{ 0 \}$, and $\r_+=(0,\infty)$. Given $n\in\n$ and $\k\in\{\r,\c\}$ we denote the Euclidean norm by $||\cdot||$, the distance between two points by ${\dist}(\cdot,\cdot)$,  the length of an arc in $\k^n$ by $\length(\cdot)$, and the absolute value of a real (or complex) number by $|\cdot|$.
Furthermore, define \(\| v \|_\infty = \max \{| v_1|, \dots, |v_n| \}\) for \(v = (v_1, \dots, v_n) \in \K^n\).

Given $p\in\k^n$ and a positive number $r>0$, we denote the ball of radius \( r \) centred at \( p \) by
\[
B(p,r):=\{x\in\k^n : ||x-p||<r \}.
\]
Let \( A \) be a subset of a topological space \( T \). We denote the closure of \( A \) in $T$ by $\overline A$, the interior of \( A \) in $T$ by $\mathring{A}$, and the boundary of \( A \) in \( T \)  by $bA$.
Given subsets $A$ and $B$ of $T$, we use the notation $A\Subset B$ to mean $\overline A\subset \mathring{B}$.

Given a smooth connected surface $S$ (with possibly nonempty boundary) and a smooth immersion $X\colon S\to\k^n$, we denote by ${\dist}_X\colon S\times S\to [0,+\infty)$ the Riemannian distance induced on $S$ by the Euclidean metric of $\k^n$ via $X$, that is:
\[
{\dist}_X(p,q):=\inf\{ \length(\alpha) : \alpha\subset S \text{ is an arc connecting $p$ to $q$} \},\quad p,q\in S.
\]
In addition, if $Q\subset S$ is a relatively compact subset of $S$, we define
\[
{\dist}_X(p,Q):=\inf \{{\dist}_X(p,q):q\in Q \}, \quad p\in S.
\]
{A divergent path on \( S \) is a continuous map \( \gamma\colon [0,+\infty) \to S \) such that for every compact set \( K \subset S \) there exists \( t_0 \in [0,+\infty) \) such that \( \gamma(t) \not \in K \) for \( t \geq t_0 \).}
An immersed open surface $X\colon S\to\k^n$, $n\ge 3$, is said to be {\em complete} if the Riemannian metric induced by ${\dist}_X$ is complete; equivalently if for any divergent path $\gamma\subset S$, we have that the Euclidean length $\length(\gamma)$ is infinite. On the other hand, an immersed open surface $X\colon S\to\k^n$ is said to be {\em proper} if for any divergent path $\gamma\subset S$, we have that $X(\gamma)$ is also a divergent path on $\k^n$.

Let \( A \) be a subset of an open Riemann surface $\Rcal$ . 
(Throughout this paper every Riemann surface will be assumed to be connected.)
We denote by $\Oscr(A)$ the space of holomorphic functions $A\to\c$ defined on some unspecified open neighbourhood of $A$ in $\Rcal$.
In particular, \( \Oscr(\Rcal)\) denotes the space of holomorphic functions \( \Rcal\to\c \).
Let  $\Ascr^r(A)$, $r\ge 0$, be the space of $\Cscr^r(A)$ functions which are holomorphic on the interior  of \( A \); here $\Cscr^r(A)$ denotes the space of differentiable functions $A\to\c$ up to order $r$. For simplicity we write  $\Ascr(A)$ for $\Ascr^0(A)$ and $\Cscr(A)$ for \( \Cscr^0(A)\).
Likewise, we define the spaces $\Oscr(A,Y)$, $\Ascr^r(A,Y)$, and $\Cscr^r(A,Y)$ of maps $A\to Y$, where $Y$ is a complex manifold.

Let  $K$ be a compact subset of an open Riemann surface \( \Rcal\). 
Define the holomorphically convex hull \( \wh K \) of \( K \) in \( \Rcal \)  by 
\[
\wh K:=\left\lbrace p\in\Rcal : |f(p)|\le \sup\limits_{q\in K} |f(q)| \text{ for any $f\in\Oscr(\Rcal)$}\right\rbrace.
\]
The set \( \wh K \) is also referred to as the \( \Oscr(\Rcal) \)-hull of \( K \).

A compact set $K\subset \Rcal$ is said to be {\em $\Oscr(\Rcal)$-convex} (sometimes called  {\emph{holomorphically convex}} or {\em Runge}) if $\wh K=K$.
Equivalent to \( \Oscr(\Rcal)\)-convexity is  the condition that \( K \) has no \emph{holes}, that is, the complement \( \Rcal \setminus K \) has no relatively compact connected components.
This is a  natural condition to consider in the context of approximation problems as evidenced by the Runge--Mergelyan Theorem \cite{Runge1885AM,Mergelyan1951DAN,Bishop1958PJM} which states that every function in $\Ascr(K)$
is uniformly approximable by entire functions if $K$ is $\Oscr(\Rcal)$-convex.
Observe that these functions are precisely those one  could hope to approximate on \( K \) because the uniform limit of holomorphic functions is holomorphic.

Now suppose that \( E \subset \Rcal \) is a closed, but not necessarily compact, subset. 
Following \cite{manne2011} we define the \emph{hull} of \( E \) to be \[ \wh E = \bigcup\limits_{j \in \N} \wh E^j,\]
where \( {\{E^j\}}_{j \in \N} \) is some compact exhaustion of \( E \).
This definition is independent of the choice of  exhaustion.
For if \( \{ E_1^j \} \) and \( \{ E_2^j \} \) were two exhaustions for \( E \) then simply interweaving the compact sets in the sequence leads to a new exhaustion \[ E_1^1 \subset E_2^{\mu_1} \subset E_1^{\nu_1} \subset E_2^{\mu_2} \subset  \dots \] which gives the same hull as  \( \{ E_1^j \}_{j \in \N} \) and \( \{ E_2^j \}_{j \in \N} \) do  since \( \wh K \subset \wh L \) whenever \( K \subset L \subset \Rcal\) are compacts. 
We say that a noncompact subset \( E \subset \Rcal \) is \( \Oscr(\Rcal )\)-convex (or holomorphically convex) if \( \wh E = E \).

For a closed set \( E \subset \Rcal \) we define 
\[
h(E)=\overline{\wh{ E} \setminus  E}.
\]
We shall say that a set \( E \) has \emph{bounded exhaustion hulls} (or {\em bounded E-hulls}) if for every compact set \( K \subset \Rcal \) we have that \( h(K \cup E) \) is compact.
In the case where \( \Rcal = \C \) this condition is equivalent to the complement \( \C \P^1 \setminus E \) being locally connected at \( \infty \).

The utility of the bounded E-hull condition for us  is that it provides us with nice  compact exhaustions of \( \Rcal \)  with respect to \( E \). Such exhaustions are used to prove our main results. More precisely,  if a subset \( E \subset \Rcal \) is \( \Oscr(\Rcal)\)-convex and has bounded E-hulls, then there is a compact exhaustion \( \{ K_j\}_{j \in \N}\) of \( \Rcal \) by \( \Oscr(\Rcal)\)-convex subsets such that \( K_j \cup E \) is \( \Oscr(\Rcal)\)-convex for every \( j \in \N \), see \cite[Lemma 2.2]{Chenoweth2018Carleman}.

A smooth Jordan arc inside of \( \Rcal \) is the image \( \gamma(I)\) of an interval \( I \subset \R \) under a smooth map \( \gamma\colon I \to \Rcal \)  which is injective except perhaps at the endpoints of \( I \); for example, a divergent arc or a closed Jordan curve.
A family \( \{A_j\}_{j \in J}\) of subsets of \( \Rcal \)  is called locally finite if for each \( p \in \Rcal \) there is a neighbourhood \( V\subset \Rcal \) of \( p \) such that \( V \cap A_j = \emptyset \) for all but finitely many \( j \in J \).

We proceed to define the type of subsets and maps we are going to deal with along the paper. 
A closed subset \( S \subset \Rcal \) is called an {\em admissible} subset if \( S = M \cup K \), where \( K \) is the union of a locally finite pairwise disjoint collection of smoothly bounded compact domains, and \( M \) is the union of a locally finite pairwise disjoint collection of smooth Jordan arcs, so that each Jordan arc intersects the boundary of each compact domain only at its endpoints, if at all, and so that all such intersections are  transverse.
This notion generalises the (compact) admissible subsets introduced by  \cite{AlarconForstneric2014IM,AlarconForstneric2015MA} to the unbounded setting. For the purpose of the paper we state the following definition.
\begin{definition} \label{def:Cadmissible}
	We will call an admissible subset which is holomorphically convex and has bounded exhaustion hulls a {\em Carleman admissible} subset. These are precisely the sets on which we can prove Theorem \ref{th:easyCarleman}.
\end{definition}

The following gives us many basic examples of Carleman admissible sets (with $K$ empty).
\begin{example}\label{ex:Rembedded}
	A smooth properly embedded image of $\R$ in an arbitrary open Riemann surface $\Rcal$ is holomorphically convex and has bounded E-hulls.
\end{example}
\noindent Note that this statement is false if we omit the word `properly', see \cite[p.141]{gaier1987} for examples.
\begin{proof}
	Let $\Rcal$ be an open Riemann surface, $\gamma\colon \R \to \Rcal$ be a smooth proper embedding, and  $M=\gamma(\R)$. For each $j \in \mathbb{N}$ define $M^j := \gamma([-j,j])$.
	Observe that $(M^j)_{j \in \mathbb{N}}$ is a compact exhaustion of $M$.
	For each $j \in \N$ we have that $\wh{M^j}=M^j$ because the complement of $M^j$ is connected and not relatively compact.
	Hence $M$ is $\Oscr(\Rcal)$-convex.
	
	Now, suppose that $K$ is a compact subset of $\Rcal$.
	Without loss of generality assume that $M \cap K \not = \emptyset$. 
	Since $\gamma$ is proper, the set $\gamma^{-1}(K)$ is compact. 
	Suppose that \( [a,b]\) is the convex hull of \(\gamma^{-1}(K) \).
	Let $K'$ be the holomorphically convex hull of $K \cup \gamma ([a,b])$ in $\Rcal$.
	Note that $K'$ is $\Oscr(\Rcal)$-convex and hence $\Rcal \setminus K'$ has no relatively compact connected components.
	
	One can easily verify using the properness and injectivity of $\gamma$ that whenever $t < a$ or $t > b$ we have that $\gamma(t) \not \in K'$.
	Let $U\subset \Rcal$ be the connected component of the complement of $K'$ containing $\gamma((-\infty,a))$.
	Clearly the set $U \setminus \gamma([a-j,a))$ is connected for every $j \in \mathbb{N}$ (one can see this using the tubular neighbourhood theorem  for example). Also, clearly $U \setminus \gamma([a-j,a))$ is not relatively compact since any divergent sequence in $U$ can be perturbed so as to avoid $\gamma([a-j,a))$.
	Hence the complement of $K' \cup \gamma([a-j,a])$ has no relatively compact connected components and therefore is $\Oscr(\Rcal)$-convex.
	We may apply the same reasoning to prove that 
	\[
	K^j:=K' \cup \gamma([a-j,a]) \cup  \gamma([b,b+j])
	\]
	is $\Oscr(\Rcal)$-convex, here we also need $ \gamma([a-j,a)) \cap  \gamma((b,b+j]) = \emptyset$ which is true by injectivity.
	
	The sequence $(K^j)_{j \in \mathbb{N}}$ is a compact exhaustion of $K' \cup M $ by holomorphically convex compact subsets. Therefore $K' \cup M$ is $\Oscr(\Rcal)$-convex and hence $h(K \cup M) \subset K'$, so the hull $h(K \cup M)$ is compact, as required.
\end{proof}

Let $\Sgot$ be a closed conical complex subvariety of $\c^n$, $n\ge 3$,
a {holomorphic} immersion \(  \mathcal{R} \to \mathbb{C}^n \) 
is said to be an {\em \( \mathfrak{S} \)-immersion} if its complex derivative \( F' \) with respect to any local holomorphic coordinates on \( \Rcal \) assumes values in \( \mathfrak{S}_* := \mathfrak{S} \setminus \{ 0 \} \).
If $L\subset \Rcal$ is a compact subset of $\Rcal$, by an $\Sgot$-immersion $L\to\c^n$ we mean a $\Ascr^1(L,\c^n)$ map in $L$ that is an $\Sgot$-immersion in $\mathring{L}$.
More generally, we can define the notion of an $\Sgot$-immersion on a Carleman admissible set as follows.
\begin{definition}\label{def:Simmersion}
	Let \( S = M\cup K \subset \mathcal{R} \) be a Carleman admissible subset, see Definition \ref{def:Cadmissible},
	a map \( F\colon S \to \mathbb{C}^n \) is said to be an {\em \( \Sgot \)-immersion} if $F$ is a map of class $\Ascr^1(S,\c^n)$ such that
	\( F|_K \) is an \( \Sgot \)-immersion and
	the derivative $F'(t)$ with respect to any local real parameter $t$ on $M$ belong to $\Sgot_*$.
\end{definition}
Strictly speaking we only need assume that \( F\) is of class  \(\mathscr{A}(S,\C^n) \) and  \( F|_K \) is an \( \Sgot \) immersion of class \( \Ascr^1(K,\c^n) \), by Mergelyan's theorem and  \cite[Lemma 3.3]{AlarconCastro2018}. 
However, making this assumption means that our notion of \( \Sgot \)-immersion is consistent with the notion introduced in \cite{AlarconForstneric2014IM,AlarconForstnericLopez2016MZ}.

\subsection{Contemporary complex analytic methods in the study of minimal surfaces}
\label{connections}

Minimal surface theory and complex analysis have enjoyed a fruitful relationship for many years due to the Weierstrass representation formula.
Recall that if $\Rcal$ is an open Riemann surface, a conformal immersion $X\colon \Rcal\to\r^n$ is minimal if and only if $X$ is a harmonic map. Moreover, denoting by $\partial$ the $\c$-linear part of the exterior differential $d$, we have that $\partial X=(\partial X_1,\ldots,\partial X_n)$ is a nowhere vanishing holomorphic $1$-form verifying
\begin{equation}\label{eq:partialX}
 \partial X_1^2 + \dots+ \partial X_n^2 =0 \quad \text{and}\quad  |\partial X_1|^2 + \dots+  |\partial X_n|^2 \neq 0.
\end{equation}
In this situation, $\Re(\partial X)$ is an exact $1$-form on $\Rcal$ and the flux map of $X$ is determined by the group morphism $\Flux_X\colon \Hcal_1(\Rcal;\z)\to\r^n$, of the first homology group of $M$ with integer coefficients, defined by
\begin{equation}\label{eq:flux}
\Flux_X(\gamma)=\int_\gamma \Im(\partial X)=-\imath\int_\gamma\partial X, \quad \gamma\in \Hcal_1(\Rcal;\z).
\end{equation}
Conversely, any holomorphic $1$-form $\Phi =(\phi_1,\ldots,\phi_n)$ 
never vanishing on $\Rcal$, verifying \eqref{eq:partialX}, and such that $\Re(\Phi)$ is an exact $1$-form on $\Rcal$ determines a conformal minimal immersion $X\colon \Rcal\to\r^n$ with $\partial X=\Phi$ via the Weierstrass representation formula
\begin{equation}\label{eq:Weierstrassformula}
X(p)=X(p_0)+2\int_{p_0}^p \Re(\Phi),\quad p\in\Rcal,
\end{equation}
for any fixed point $p_0\in \Rcal$ and initial condition $X(p_0)\in\r^n$. For a standard reference see \cite{Osserman-book}.
More recently, tools from holomorphic elliptic geometry have emerged and found their way into contemporary minimal surface theory, see the survey \cite{AlarconForstneric2015Survey}.

\subsection{Elliptic holomorphic geometry}
\label{Oka}

A Stein manifold, introduced in \cite{stein1951}, is a complex manifold with many holomorphic functions on it. More precisely, we say that \( X \) is Stein if
\begin{enumerate}[\rm (i)]
	\item For every \( x, y \in X \), \( x \neq y\), there exists a function \( f \in \mathscr{O}(X)\) such that \( f(x)\not = f(y)\).
	\item For every compact set \( K \subset X  \) the hull  \( \wh K \) is also compact.
\end{enumerate}
A Riemann surface is Stein if and only if it is noncompact, \cite[Corollary 26.8]{Forsterbook}.

The generalisation of Runge's theorem to Stein manifolds is called the Oka-Weil approximation theorem. 
In his seminal paper Gromov \cite{gromov1989} considered the Oka-Weil approximation theorem as expressing a property of the target, in modern language, the approximation property.
A complex manifold \( Y \) is said to have the approximation property if for every Stein manifold \( X \) with an \( \Oscr(X)\)-convex subset \( K \), every continuous map \( X \to Y\) which is holomorphic on a neighbourhood of \( K \) can be approximated uniformly on \( K \) by holomorphic maps \( X \to Y \). 

An Oka manifold is a complex manifold which satisfies the approximation property. There are more than a dozen nontrivially equivalent  characterisations of Oka manifolds; we chose to mention this characterisation because we are considering approximation problems.
For a comprehensive introduction to Oka theory, see Forstneri\v{c}'s book  \cite{forstneric2017} 
or Forstneri\v{c} and L\'{a}russon \cite{ForstnericLarusson2011NY}.

\section{Carleman approximation by directed immersions}\label{sec:Carleman}

We start this section with one of the main technical results in paper. 
The following result is the aforementioned gluing lemma that will be needed in the induction procedure. This gluing lemma together with Mergelyan's theorem will enable us to approximate well on a noncompact sets by directed holomorphic immersions.

\begin{lemma}\label{lem:gluing}
Let  \( \mathfrak{S} \) be a  closed conical complex subvariety which is not contained in a hyperplane of \( \mathbb{C}^n \) and  such that \( \mathfrak{S}_* := \mathfrak{S} \setminus \{ 0 \} \) is smooth and connected. 
Let \( \mathcal{R} \) be an open Riemann surface,  let \( L \subset \Rcal \) be a compact domain of \( \Rcal\)  and 
let $S=M\cup K$ be a Carleman admissible subset, see Definition \ref{def:Cadmissible}, such that $K\cap bL=\emptyset$ and $M$ intersects
\( L \) transversely (if at all), and let \( \epsilon \colon M\cup K \to \mathbb{R}_+ \) be a positive continuous function.
Suppose that \( F\colon M\cup K \to \mathbb{C}^n\) is a \( \mathfrak{S} \)-immersion.
If there is an \( \mathfrak{S} \)-immersion \( \widehat F \colon \mathcal{R} \to \mathbb{C}^n\) such that 
\begin{align}\label{altj}
\| F(p)- \widehat F(p)\| < \epsilon(p) \quad \text{for every \( p \in (M\cup K) \cap L \)},
\end{align}
then 
there exists a  \( \mathfrak{S} \)-immersion \( \wt F\colon M\cup K\cup L \to \mathbb{C}^n \) such that 
\(\widetilde F|_{L} = \widehat F|_{L}
\)
and 
\begin{align}\label{eq:solgluing}
	\| F(p)- \widetilde F(p)\| < \epsilon(p) \quad  \text {for all \( p \in M\cup K \).}
\end{align}
Furthermore, we may ensure that \( \wt F =   F \) outside of an arbitrarily small neighbourhood of  \( L \). 
\end{lemma}
\noindent Under the additional hypothesis that \( L \) is \( \Oscr(\Rcal) \)-convex the following proof, and hence the lemma, holds for \( \widehat F \) merely defined on \( L \). For we may use Mergelyan theorem for \( \Sgot \)-immersions to approximate \( \widehat F \) by an \( \Sgot \)-immersion defined on all of \( \Rcal \).
\begin{proof}
	Let \( \widehat F\colon \mathcal{R}\to \mathbb{C}^n \) be an arbitrary  \( \mathfrak{S} \)-immersion satisfying \eqref{altj}. Naively we could define \( \widetilde F\colon M\cup K\cup L  \to \mathbb{C}^n\) as
\begin{align*}
\wt F(p) = \begin{cases}
					\widehat F(p) & \text{if } p \in L, \\
					F(p)  &\text{if } p \in (M\cup K) \setminus L. 
				\end{cases}
\end{align*}
This function has a finite number \( r\in\z_+ \) of \emph{trouble points}, that is, points which are discontinuities or critical points, since $K\cap bL=\emptyset$. Clearly, all trouble points are contained in \( M \cap b L \). 

If \( r = 0 \), then we are done. If \( r > 0 \), then we make \( \wt F \) into the desired \( \mathfrak{S} \)-immersion by gluing near each of the trouble points as described below.

Let \( p \) be a fixed trouble point of \( \widetilde F \). Let \( U\subset \Rcal \) be a small neighbourhood of \( p \) not intersecting \( K \) or any other trouble point.
Let \( \gamma\colon (-1,2) \to M \cap U \subset\mathcal{R}\) be an embedded arc in \( \mathcal{R}\) such that \( \gamma(1) = p \), \( \gamma((-1,1]) \subset \mathcal{R} \setminus L \) and \( \gamma((1,2)) \subset L \setminus K\).

Observe that by \eqref{altj}
\begin{align*}
	\|\widehat F(p) - F(p) \| < \epsilon(p)
\end{align*}
and therefore by continuity, there is \( \delta_1 > 0 \) small enough such that
\begin{align}
\rho_1 := \| \widehat F(p) - F(p) \| < \epsilon(\gamma(t)) & \quad \text{for }  t \in [1-\delta_1,1].
\end{align}
Let \(\rho_2=\frac{1}{2} \big(\rho_1+ \min\limits_{t \in [1-\delta_1,1]} \epsilon(\gamma(t))\big) \).
Note that \( \rho_ 2> \rho_1 > 0 \).
Choose \( \delta_1 >\delta_2 > 0 \) to be sufficiently small that 
\begin{align*}
\|  \widehat F(\gamma(1-\delta_2)) -  \widehat F(p) \| &<\rho_2-\rho_1, \ \text{ and} \\
\|  F(\gamma(t)) -  F(p) \| < \rho_2-\rho_1, \quad &\text{whenever } t \in [1-\delta_2,1]. 
\end{align*}
Reparameterising if necessary we may suppose that \( \delta_2 =1 \).
Let \(  N = B(F(p),\rho_2)  \). Observe that \( \widehat F(\gamma(1-\delta_2)), F(p) \in N \) by our choices above.
By \cite[Lemma 3.3]{AlarconCastro2018} there is \( G\colon \gamma([0,1]) \to \mathbb{C}^n \) such that
\begin{itemize}
		\item \( \displaystyle \frac{\partial}{\partial t} (G(\gamma(t))) \in \mathfrak{S}_* \),
		\smallskip
		\item \( G( \gamma(0)) = F(\gamma(0))\) and \( G(\gamma(1))  =\widehat F(\gamma(1))\),
		\smallskip
		\item \( \displaystyle\frac{\partial}{\partial t}(G(\gamma(t)))|_{t=0} =  \frac{\partial}{\partial t}( F(\gamma(t)))|_{t=0} \in \mathfrak{S}_* \),
		\smallskip
		\item \(  \displaystyle\frac{\partial}{\partial t}(G(\gamma(t)))|_{t=1} = \frac{\partial}{\partial t}(\wh F(\gamma(t)))|_{t=1} \in~\mathfrak{S}_*,\)
		\smallskip
		\smallskip 
		\item \( G( \gamma(t)) \subset N \) for every \( t \in [0,1] \).
	\end{itemize}
Redefining \( \wt F \) on \( \gamma([0,1])  \) as \( G \)  gives us a map \( M\cup K\cup L \to \mathbb{C}^n \) with one fewer trouble point.
Repeating this gluing procedure \( r \) times yields the desired \( \widetilde F \).
\end{proof}

We now prove the following result which is an analogue of the classical Carleman Theorem for directed immersions.

\begin{theorem}[Carleman Theorem for directed immersions]\label{th:Carleman}
Let \( \Sgot \) be an irreducible closed conical complex subvariety of \( \C^n \), \( n \geq 3\), which is contained in no hyperplane and such that \( \Sgot_* = \Sgot \setminus \{0\} \) is a (connected) Oka manifold. 
Let $\Rcal$ be an open Riemann surface.
Let $S=M\cup K\subset \Rcal $ be a Carleman admissible subset.
Let $\Lambda\subset\Rcal\setminus M$ be a discrete closed subset of $\Rcal$ and let $\Omega_p \ni p$ for  $p\in\Lambda$ be small pairwise disjoint compact neighbourhoods.
Set $\Omega=\cup_{p\in\Lambda}\Omega_p$.
	Given a map $k\colon \Lambda\to\z_+$, a positive function $\epsilon\colon M\cup K\to\r_+$ and an $\Sgot$-immersion $F\colon M\cup K\cup\Omega\to\c^n$,  there exists an $\Sgot$-immersion $\wt F\colon \Rcal\to\c^n$ that verifies the following properties:
	\begin{enumerate}[\rm (i)]
		\item $\| \wt F(p)- F(p)\|<\epsilon(p)$ for any $p\in M\cup K$.
		\item $\wt F- F$ has a zero of multiplicity $k(p)$ at any $p\in\Lambda$.
		\item If $F|_\Lambda$ is injective, then $\wt F$ is injective.
	\end{enumerate}
\end{theorem}
\begin{proof}
Consider an exhaustion $\{K_j\}_{j\in\n}$ of $\Rcal$ by compact \( \Oscr(\Rcal) \)-convex subsets  such that
\begin{equation}\label{eq:exhaustion}
	 K_1\Subset K_2 \Subset  K_3 \Subset \cdots\Subset\bigcup_{j\in\n}K_j=\Rcal,
\end{equation}
and that $(K \cup \Omega)\cap bK_j=\emptyset$ for any $j\in\n$.
Recall that since $M$ has bounded E-hulls it is possible to find the sequence of  $\Oscr(\Rcal)$-convex  compacts such that
$(M\cup K )\cap K_j$ is again a $\Oscr(\Rcal)$-convex compact subset of \( \Rcal \)
and such that $M$ intersects $bK_j$ transversely a finite number of times
for each $j\in\n$.
	
Set $F_0:=F\colon M\cup K\cup\Omega\to\c^n$ and \( K_0 = \emptyset \). To prove the theorem we are going to construct a sequence of $\Sgot$-immersions $F_j\colon M \cup K \cup \Omega \cup K_j\to\c^n$ with the following properties for any $j\in\n$:
\begin{enumerate}[\rm (I$_j$)]	
	\item $\| F_{j}(p)- F_{j-1}(p)\|<1/2^j$ for any $p\in K_{j-1}$.
	\smallskip
	\item  $\| F_{j}(p)- F_{j-1}(p)\|<{\epsilon(p)}/{2^j}$ for any $p\in M\cup K$.
	\smallskip
	\item $F_{j}- F_{j-1}$ has a zero of multiplicity (at least) $k(p)$ at each point $p\in\Lambda$.
	\smallskip
	\item If $F|_\Lambda$ is injective, then $F_j|_{K_j}$ is also injective.
\end{enumerate}

Assume first that we have constructed such a sequence. Then properties {\rm (I$_j$)} for $j\in\n$ ensure that there exists a limit map $\wt F\colon \Rcal\to\c^n$, given by
\[
\wt F=\lim\limits_{j\to+\infty} F_j,
\]
and that this map is an $\Sgot$-immersion. 
Properties {\rm (II$_j$)} guarantee condition {\rm (i)}, 
whereas condition {\rm (ii)} follows from properties {\rm (III$_j$)}.
Finally, if $F|_\Lambda$ is injective, then properties {\rm (IV$_j$)} imply condition {\rm (iii)}.

Now we will recursively construct the sequence \( \{ F_j \}_{j \in \N } \).
It is clear that $F_0=F$ verifies the base of the induction. 
Assume now that we have an $\Sgot$-immersion $F_{j-1}\colon M\cup K\cup\Omega \cup K_{j-1}\to\c^n$ satisfying  properties {\rm (I$_{j-1}$)}--{\rm (IV$_{j-1}$)} for some \( j \geq 1 \).

Let $\epsilon'\colon M\cup K \cup K_{j-1}\to\r_+$  be the positive function defined by
\begin{equation}\label{eq:epsilon1}
\epsilon'(p)=	\begin{cases}
\frac{1}{2^j}\epsilon(p) &\quad\text{if } p\in (M \cup K)\setminus K_{j-1},\\
\frac{1}{2^j}\min\{1,\epsilon(p)\} &\quad \text{if } p\in (M\cup K)\cap K_{j-1},\\
\frac{1}{2^j} &\quad\text{if }p\in K_{j-1}\setminus (M\cup K).
\end{cases}
\end{equation}
Since  \( (M\cup K)\cap K_{j}\) is an $\Oscr(\Rcal)$-convex compact admissible subset of $\Rcal$, by \cite[Theorem 1.3]{AlarconCastro2018}, given numbers
\[
0<\delta<\min\{ \epsilon'(q): q \in  (M\cup K)\cap K_{j}\}, \quad \text{and}\quad \max\{k(p):p\in\Lambda\cap K_j \}\in\z_+,
\]
we may approximate $F_{j-1}$ on  \( (M\cup K\cup K_{j-1})\cap K_{j}\) by an $\Sgot$-immersion $\wh F\colon \Rcal\to\c^n$ that satisfies the following:
\begin{enumerate}[\rm (a)]
	\item $\| \wh F(p)-F_{j-1}(p) \|<\delta<\epsilon'(p)$ on \( (M\cup K\cup K_{j-1})\cap K_{j}\),
	\item $\wh F-F_{j-1}$ has a zero of multiplicity (at least) $k(p)$ at each point of $p\in\Lambda$.
	\item If \( F_{j-1}|_{\Lambda}\) is injective, $\wh F$ is injective.
\end{enumerate}

Next, we can apply Lemma \ref{lem:gluing} to a positive continuous function $ M\cup  K \cup K_{j-1}\to\r_+$ less than \( \epsilon' \) and the maps $F_{j-1}\colon M\cup K \cup \Omega \cup K_{j-1}\to \c^n$ and $\wh F\colon \Rcal\to\c^n$ to obtain an $\Sgot$-immersion $F_{j}\colon M\cup K \cup \Omega \cup K_{j}\to\c^n$ such that
\begin{enumerate}[\rm (c)]
	\item[\rm (d)] $F_{j}=\wh F$ on $K_{j}$.\smallskip
	\item[\rm (e)] $\|F_{j}(p)-F_{j-1}(p)\|<\frac{1}{2^j}$ for any $p\in K_{j-1}$.\smallskip
	\item[\rm (f)] $\|F_{j}(p)-F_{j-1}(p)\|<\frac{\epsilon(p)}{2^{j}}$ for any $p\in M\cup K$.\smallskip
	\item[\rm (g)] \( F_{j} = F_{j-1} \) on \( \Omega \setminus K_{j} \).
\end{enumerate}
Property {\rm (g)} is  a consequence of the fact that \( \Omega \) and \( M \cup K \) are separated,  hence we may simply define \( F_{j} =F_{j-1} \) on \( \Omega \setminus K_j \).
Then, we claim that the $\Sgot$-immersion $F_{j}$ verifies the conditions {\rm (I$_{j}$)}, {\rm (II$_{j}$)}, {\rm (III$_{j}$)}, and {\rm (IV$_{j}$)}. Indeed, property {\rm (e)} gives us {\rm (I$_{j}$)} whereas {\rm (II$_{j}$)} follows from {\rm (f)}. Properties {\rm (b)}, {\rm (d)}, and {\rm (g)} imply that condition {\rm (III$_{j}$)} holds. Finally, if we assume that $F|_\Lambda$ is injective, then condition {\rm (IV$_j$)} is a consequence of {\rm (IV$_{j-1}$)} and properties {\rm (c)} and {\rm (d)}.

This finishes the construction of the sequence and concludes the proof of the theorem.
\end{proof}

\begin{remark}\label{rem:avoidtopology} \label{fattening}
{The approximation and interpolation result  \cite[Theorem 1.3]{AlarconCastro2018} 
is stated for smoothly bounded compact domains but also applies for general admissible subsets. For we may apply \cite[Proposition 5.2]{AlarconCastro2018} in order to get an $\Sgot$-immersion defined on a smoothly bounded compact domain that contains the  initial admissible subset and then \cite[Theorem 1.3]{AlarconCastro2018}  gives the desired $\Sgot$-immersion.

Notice also that although \cite[Theorem 1.3]{AlarconCastro2018} gives us jet-interpolation of a fixed order \( k \)  at each point of \( \Lambda \), slight modifications in the proof of this result would yield jet-interpolation to an arbitrary order  $k(p)\in\z_+$ depending on the point $p\in\Lambda$. However,  it is enough for our reasoning to apply the original result at each step of the induction to the largest order for the points treated up to this step.
}
\end{remark}

\begin{remark}\label{rem:necesity}
{	Following the proof of  \cite[Theorem 1.2]{magnusson2016} we   can see that the conclusion of Theorem \ref{th:Carleman} implies \( S \) has bounded  E-hulls and therefore this hypothesis is necessary.
	Suppose \( S \) does not have bounded E-hulls.  Then there is some compact set \( L \) so that \( h(L \cup S )\) is non-compact.
	Therefore there exists a sequence \( \Lambda =\{ x_j\}_{j \in \N}\subset h(L \cup S) \) without a point of accumulation.
	Let \(  \Omega \supset \Lambda \) be the union of disjoint coordinate neighbourhoods  which each contain precisely one element of \( \Lambda \)  such that the closure of \( \Omega \) does not intersect \( S\).
	
	Define a map \( F\colon S \cup \Omega \to \C^n \) as follows.
	On the connected component of \( \Omega \) containing \( x_j \) we define \( F \) as some \( \Sgot\)-immersion with \( \|F(x_j)\|>j\). On each connected component of \( K \) we define \( F \) to be some \( \Sgot \)-immersion such that \( \| F(x)\| < 1/4\) for every \( x \in K \)  and such that \( F \) extends as an \( \Sgot\)-immersion to a small neighbourhood of \( K \). Finally using the gluing Lemma we extend to a map on \( S \) with \( \| F(x)\| \leq 1/2 \) for every \( x \in S \).  We now apply Theorem  \ref{th:Carleman} to obtain an \( \Sgot \)-immersion \( \wt F\colon \Rcal \to \C^n \) such that \( \| F(p)\| < 1\) for every $p\in S$ and \( \| F(x_j)\| > j\) for every \( j \in \N\). 
	Let \( C = \max_{x \in L} \| F(x)\| \).
	Given a natural number \( j > C \) we have that \( \|F(x_j)\|> j> C\) which contradicts the assumption that \( \Lambda \subset h(L \cup S) \). Therefore \( S \) must have bounded exhaustion hulls for the conclusion of Theorem \ref{th:Carleman} to hold.}
\end{remark}

Theorem \ref{th:Carleman} gives an approximation result for directed immersions in {Carleman admissible subsets}, see Definition \ref{def:Cadmissible}, but we may prove a similar approximation result for more general subsets.

\begin{definition}\label{def:totallyrealset}
	We say that \( M \subset \Rcal \) is a totally real set (of class \( \Cscr^\infty \)) if \( M \) is closed and locally contained in a smooth real submanifold of \(\Rcal \), 
	that is, for every point \( x \in M \) there exist an open neighbourhood \( U \) of \( x \) and a 
	closed \( \mathscr{C}^\infty \)   real submanifold \( \wt M \subset U \) such that \(  M \cap U \subset \wt M \).
\label{def:stratified}
A closed subset $M\subset \Rcal$ is called a stratified totally real set if $M$ is the increasing union $M_0\subset M_1\subset \cdots\subset M_k=M$ of closed subsets of $\Rcal$, such that ${M_j\setminus M_{j-1}}$ is a totally real set for any $j=0,\ldots,k\in\n$, where \( M_{-1} := \emptyset \).
	
	We will refer to the set \( \overline{{M_j\setminus M_{j-1}}} \)   for  \( j =0, \dots, k \) as the strata of \( M \) with respect to the stratification \( \{ M_j\}_{0 \leq j \leq k} \). 
	For convenience, whenever we say $M$ is a stratified set, we will assume that we have fixed a stratification like the one above so that we may refer to the strata of \( M \) without having to introduce more notation.
\end{definition}

Let \( M \subset \Rcal \) be a totally real set. We say that \( F\colon M \to \C^n \) is an \( \Sgot \)-immersion if locally \( F \) looks like the restriction of an \( \Sgot \) immersion, that is, for each \( p \in M \) there is a neighbourhood \( U \ni p \), a 
closed \( \mathscr{C}^\infty \)   real submanifold \( \wt M \subset U \) such that
\[
	M \cap U  \subset \wt M \subset U,
\]
and an \( \Sgot \)-immersion \( \wt F \colon \wt M \to \C^n \) such that \(F(z) = \wt F(z) \) for every \( z \in M \cap U \). 
An \( \Sgot \)-immersion on a stratified totally real set \( M \subset \Rcal \) is simply a map \( F\colon M \to \C^n \) such that $F$ restricted to a strata of \( M \) is an \( \Sgot \)-immersion in the sense described above.
If $M\cup K\subset \Rcal$ is a subset where \( K \) is a locally finite union of smoothly bounded pairwise disjoint connected compact subsets and \( M \) is a stratified totally real subset then a map \( F\colon M \cup K \to \C^n \) is said to be an $\Sgot$-immersion if \( F|_{M}\) and \( F|_{K}\) are \( \Sgot\)-immersions.

\begin{corollary}\label{co:stratified}
	Theorem \ref{th:Carleman} holds if $M$ is a stratified real $\Oscr(\Rcal)$-convex subset of $\Rcal$ for which the intersection of any two strata is discrete, and $S=M\cup K$ has still bounded E-hulls.
\end{corollary}
\begin{proof}
First observe that the proof of Theorem \ref{th:Carleman} goes through in  precisely the same way if we assume \( M \) is a totally real set rather than a finite disjoint union of closed real submanifolds.
For in the induction procedure we could simply choose a compact exhaustion of \( \Rcal \) with many  compact sets so that  \( M \cap (K_j \setminus K_{j-1}) \) is contained within a real submanifold for each \( j \in \N \) and then proceed as we did before.

Now let us consider the case where $M$ is a stratified real subset of $\Rcal$ for which the intersection of any two strata is discrete.
	
	Let \(Q\) be the set of points at which two (or more) strata intersect.
By assumption \( Q \) is discrete.
	Around each point $q\in Q$ consider a small simply connected compact neighbourhood $K_q\subset \Rcal$ such that $M \cup K_q$ is a closed $\Oscr(\Rcal)$-convex subset and such that $K_q\cap \Omega=\emptyset$.
	Then the subset defined by
	\[
		S':=M\cup K\cup \bigcup_{q\in Q}K_q
	\]
	is a Carleman admissible subset. 
	Next, we extend $F$ to $S'$ being an $\Sgot$-immersion $S'\cup\Omega\to\c^n$. For instance, we may choose an appropriate disc at each $K_q$ and then use Lemma \ref{lem:gluing} to glue around the {\em trouble points} of $M\cap K_q$. If necessary, shrink the subset $K_q$ for any $q\in Q$ and recall that since $M$ is a stratified subset then the number of {\em trouble point} at each $K_q$ is finite.
	Finally Theorem \ref{th:Carleman} provides the desired $\Sgot$-immersion $\wt F\colon \Rcal\to\c^n$.
\end{proof}

\section{Carleman approximation by complete immersions}\label{sec:completeness}

This section is dedicated to proving Theorem \ref{th:completeness2} from which Theorem \ref{th:completeness} is a direct consequence. As mentioned already, the techniques based on the ideas of Jorge and Xavier \cite{JorgeXavier1980AM} and Nadirashvili \cite{Nadirashvili1996IM} which are usually used in order to get completeness  do not apply in our particular case; therefore we develop a new approach to ensure completeness of the solutions.

\begin{theorem}\label{th:completeness2}
	The $\Sgot$-immersion $\wt F\colon \Rcal\to\c^n$, $n\ge 3$, constructed in  Theorem \ref{th:Carleman} may be chosen to be complete 
	provided that $\Sgot$ satisfies the hypotheses of Theorem \ref{th:completeness}.
\end{theorem}

\begin{proof}
	Consider an exhaustion $\{K_j\}_{j\in\n}$ of $\Rcal$ by compact \( \Oscr(\Rcal) \)-convex subsets as in the proof of Theorem \ref{th:Carleman}. Recall that the sequence verifies \eqref{eq:exhaustion} and that $(K \cup \Omega)\cap bK_j=\emptyset$ for any $j\in\n$. In addition, since $M$ has bounded E-hulls we may assume that $(M\cup K )\cap K_j$ is again an $\Oscr(\Rcal)$-convex compact subset of \( \Rcal \) and also that $M$ intersects $bK_j$ transversely a finite number of times for each $j\in\n$. 
	
	Fix any point $p_0\in \mathring K_1\setminus \Omega$. Set $F_0:=F\colon M\cup K\cup\Omega\to\c^n$ and \( K_0 = \emptyset \). To prove the theorem we construct a sequence of $\Sgot$-immersions $F_j\colon M \cup K \cup \Omega \cup K_j\to\c^n$ with the following properties for any $j\in\n$:
	\begin{enumerate}[\rm (I$_j$)]	
		\item $\| F_{j}(p)- F_{j-1}(p)\|<1/2^j$ for any $p\in K_{j-1}$.
		\smallskip
		\item  $\| F_{j}(p)- F_{j-1}(p)\|<{\epsilon(p)}/{2^j}$ for any $p\in M\cup K$.
		\smallskip
		\item $F_{j}- F_{j-1}$ has a zero of multiplicity (at least) $k(p)$ at each point $p\in\Lambda$.
		\smallskip
		\item If $F|_\Lambda$ is injective, then $F_j|_{K_j}$ is also injective. 
		\smallskip
		\item The distance ${\dist}_{F_j}(p_0,bK_j)>j-1$.
	\end{enumerate}
	
	Clearly, reasoning as in Theorem \ref{th:Carleman}, properties {\rm (I$_j$)}--{\rm (IV$_j$)} guarantee the existence of an $\Sgot$-immersion $\wt F\colon \Rcal\to\c^n$ that verifies Theorem \ref{th:Carleman}. Furthermore, due to condition {\rm (V$_j$)}, any divergent path starting at $p_0$ has infinite length, and hence $\wt F$ is complete.

	To construct the sequence and prove the theorem we reason inductively. Assume that we have constructed an $\Sgot$-immersion $F_{j-1}\colon M \cup K \cup \Omega \cup K_{j-1}\to\c^n$ that satisfies properties {\rm (I$_{j-1}$)}--{\rm (V$_{j-1}$)} for some $j\in\n$ and let us construct $F_j$.

	First, by Theorem \ref{th:Carleman} we may assume that $F_{j-1}$ is defined in all $\Rcal$.
	Next, consider an $\Oscr(\Rcal)$-convex compact subset $L$ with $K_{j-1}\Subset L\Subset K_j$ such that $(\Omega\cup K)\cap K_j\subset \mathring L$, $L$ is a strong deformation retract of $K_j$,  and 
	\begin{equation}\label{eq:goodintersections}
	\begin{array}{c}
	\text{		$M\cap (K_j\setminus \mathring L)$ consists of finitely many pairwise disjoint Jordan arcs}\\
	\text{connecting $bL$ to $bK_j$ and meeting each boundary transversely.}
	\end{array}
	\end{equation}
	This may be achieved by enlarging the subset $L$ as much as necessary.
	
	Then $\chi(K_j\setminus \mathring L)=0$ and so $K_j\setminus \mathring L$ consists of a finite union of pairwise disjoint annuli.
	To make the proof more readable we assume $K_j\setminus \mathring L$ is connected, hence a single annulus. If that is not the case, just apply the following reasoning to each of the annuli in $K_j\setminus \mathring L$. Set $A:=K_j\setminus \mathring L$.
	
	We may assume that if $\alpha$ is an arc contained in $M\cap A$ that connects $bL$ to $bK_j$ (take into account \eqref{eq:goodintersections}), then the length of the first coordinate projection of $\alpha$ is as large as desired. Concretely we would assume that 
	if $\alpha$ is an arc connecting $bL$ to $bK_j$ then 
	\begin{equation}\label{eq:length1}
	\length\big(\pi_1\circ F_{j-1}(\alpha)\big)
	>1.
	\end{equation}
	See Figure \ref{fig:complete}.
	If this were not the case, then we could approximate $F_{j-1}|_{M \cap A}$ in the $\Cscr^0(M\cap A)$-topology by a sufficiently long  arc,   glue with $F_{j-1}$ at the endpoints using \cite[Lemma 3.3]{AlarconCastro2018}, and  apply Theorem \ref{th:Carleman} to obtain a globally defined map with the desired property.

	By continuity of $F_{j-1}$, there exists a neighbourhood $U\subset A$ of $bL\cup (M\cap A)$ such that any arc $\alpha\subset U$ connecting $bL$ to $bK_j$, 
	still verifies \eqref{eq:length1}.
	Let $D\subset A\setminus M$ be a collection of pairwise disjoint simply connected compact subsets such that if $\alpha\subset A$ is an arc connecting $bL$ to $bK_j$ that does not intersect $D$ then $\alpha\subset U$ and so it verifies \eqref{eq:length1}.
	See Figure \ref{fig:complete}.
	\begin{figure}[h]
		\centering
		\includegraphics[width=13.5cm]{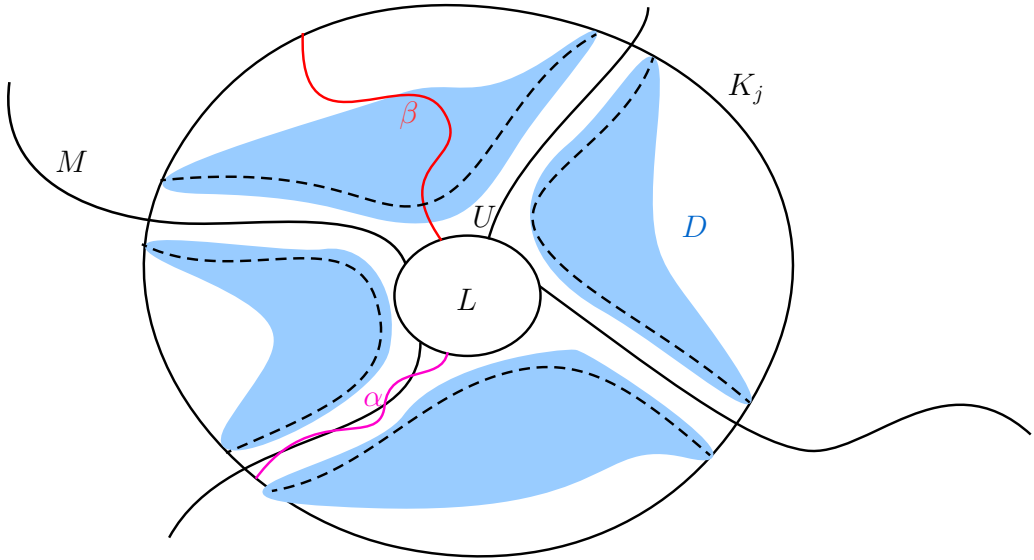}
		\caption{The set $D$ with $4$ connected components and an arc $\beta$ intersecting $D$ and an arc $\alpha$ that does not intersect $D$.}
		\label{fig:complete}
	\end{figure}
	
	Notice that since $D\subset A\setminus M$ is a union of pairwise disjoint simply connected compact subsets then $M\cup L\cup D$ is $\Oscr(\Rcal)$-convex. Let $H\colon (M\cup L\cup D)\cap K_j\to\c^n$ be an $\Sgot$-immersion such that
	\begin{enumerate}[\rm (a)]
		\item $H=F_{j-1}$ on $(M\cup L)\cap K_j$,
		\item the first coordinate $\pi_1(H)=\pi_1(F_{j-1})$ on $D$,
	\end{enumerate} 
and for $2\le a \le n$ we have $\pi_a(H)=\pi_a(F_{j-1})+\zeta$ where $\zeta\in\c$ is a constant  sufficiently large such that if $\alpha\subset A$ is an arc connecting $bL$ to $bK_j$ that intersects $D$ then
		\begin{equation}\label{eq:length}
		\length \big(H(\alpha)\big)
		>1.
		\end{equation}
	
	Observe that the first coordinate of $H$ extends to $K_j$ since $\pi_1(H)=\pi_1(F_{j-1})$. Consider the positive function $\rho\colon M\cup L \cup D \to\r_+$ defined by
	\begin{equation}\label{eq:rho}
	\rho(p)=	\begin{cases}
	\frac{1}{2^j}\epsilon(p) &\quad\text{if } p\in (M\cup L\cup D)\setminus K_{j-1},\\
	\frac{1}{2^j}\min\{1,\epsilon(p)\} &\quad \text{if } p\in (M\cup L\cup D)\cap K_{j-1},\\
	\frac{1}{2^j} &\quad\text{if }p\in K_{j-1}\setminus (M\cup L\cup D).
	\end{cases}
	\end{equation}
	Given a positive number $0<\delta<\rho(p)$ for any $p\in (M\cup L\cup D)\cap K_j$, \cite[Proposition 5.2]{AlarconCastro2018} applied to the data
	\[
		(M\cup L\cup D)\cap K_j\subset K_j,\quad \Lambda\cap K_j\subset L,\quad\text{and}\quad \max\{k(p):p\in\Lambda \}\in\z_+
	\]
	and the $\Sgot$-immersion $H$, provides an $\Sgot$-immersion $\wh H\colon K_j\to\c^n$ that verifies the following properties:
	\begin{enumerate}[\rm (a)]
		\item[\rm (c)] $\|\wh H(p)-H(p)\|<\delta<\rho(p)$ for any $p\in (M\cup L\cup D)\cap K_j$.
		\item[\rm (d)] $\wh H-H$ has a zero of multiplicity $k(p)\in\n$ for any $p\in \Lambda\cap K_j\subset L$.
		\item[\rm (e)] $\pi_{1}(\wh H)=\pi_1(H)=\pi_1(F_{j-1})$ on $K_j$.
	\end{enumerate}
	Note that \cite[Proposition 5.2]{AlarconCastro2018} provides jet-interpolation of fixed order so we have applied the result to the integer $\max\{k(p):p\in\Lambda \}\in\z_+$ to ensure property {\rm (e)}.
	Moreover, by \cite[Theorem 1.3]{AlarconCastro2018} we may suppose also that
	\begin{enumerate}[\rm (a)]
		\item[\rm (f)] If $H|_\Lambda$ is injective, then $\wh H$ is also injective.
	\end{enumerate}
	
	Finally, by properties {\rm (a)} and {\rm (c)} we have that Lemma \ref{lem:gluing}, applied to glue $\wt H$ with $F_{j-1}$, provides an $\Sgot$-immersion $F_{j}\colon M \cup K \cup \Omega \cup K_j\to\c^n$ that satisfies the following properties.
	\begin{enumerate}[\rm (a)]
		\item[\rm (g)] $F_j=\wh H$ on $K_j$.
		\item[\rm (h)] $\|F_j(p)-  F_{j-1}(p)\|<\rho(p)$ for any $p\in M \cup K \cup \Omega\cup K_j$.
	\end{enumerate}
	Note that since $(\Omega\cup K)\setminus K_j$ and $K_j$ are separated we have extended $F_j=F_{j-1}$ on $(\Omega\cup K)\setminus K_j$.
	
	We claim that the $\Sgot$-immersion $F_j$ recently constructed satisfies the desired conditions {\rm (I$_j$)}--{\rm (V$_j$)}. Indeed, condition {\rm (I$_j$)} and {\rm (II$_j$)} follows from {\rm (h)} and \eqref{eq:rho}. Properties {\rm (a)}, {\rm (d)}, and {\rm (g)} imply that condition {\rm (III$_j$)} holds. 
	On the other hand, if $F$ is injective then it is clear that condition {\rm (IV$_j$)} is a consequence of {\rm (f)}, {\rm (g)}, {\rm (a)}, and {\rm (IV$_{j-1}$)}.

	Finally, we will verify that the completeness condition {\rm (V$_j$)} is a consequence of {\rm (V$_{j-1}$)}, properties {\rm (b)}, {\rm (c)} and {\rm (e)}, and equations \eqref{eq:length1} and \eqref{eq:length}. 	
 To this end, take an arc $\gamma\subset K_j$ connecting $p_0$ to $bK_j$. There exists a connected subarc $\alpha\subset \gamma$ that connects $bL$ to $bK_j$ and $\alpha\subset K_j\setminus \mathring{L}$. 
	
	If $\alpha\cap D\neq\emptyset$, then provided that the approximation in {\rm (h)} is good enough we have that:
	\begin{eqnarray*}
	\length(F_j(\gamma)) & = & \length\big(F_j(\gamma\setminus\alpha)\big) +\length\big(F_j(\alpha)\big) \\
		& \stackrel{\text{{\rm (h)}\&$(K_{j-1}\subset L)$}}{\ge} & {\dist}_{F_{j-1}}(p_0,bK_{j-1}) +\length\big(F_j(\alpha)\big)\\
		& \stackrel{\text{\rm (IV$_{j-1}$)}}{>} & (j-2) + \length\big(F_j(\alpha)\big)\\
		& \stackrel{\text{{\rm (g)}\&{\rm (c)}\&\eqref{eq:length}}  }{>} & (j-2)+1 = j-1.
	\end{eqnarray*}
	On the other hand, if $\alpha\cap D=\emptyset$ then $\alpha\subset U$, hence:
	\begin{eqnarray*}
	\length(F_j(\gamma)) & = & \length\big(F_j(\gamma\setminus\alpha)\big) +\length\big(F_j(\alpha)\big) \\
	& \stackrel{\text{{\rm (h)}\&$(K_{j-1}\subset L)$}}{\ge} & {\dist}_{F_{j-1}}(p_0,bK_{j-1}) +\length\big(F_j(\alpha)\big)\\
	& \stackrel{\text{\rm (IV$_{j-1}$)}}{>} & (j-2) + \length\big(\pi_1\circ F_j(\alpha)\big)\\
	& \stackrel{\text{{\rm (a)}\&{\rm (c)}\&{\rm (g)}\&\eqref{eq:length1}}  }{>} & (j-2)+1 = j-1.
	\end{eqnarray*}
	This proves {\rm (V$_j$)} and concludes the proof of the existence of the sequence.
\end{proof}

\section{Carleman approximation by proper immersions}\label{sec:properness}

This section is devoted to proving Theorem \ref{th:properness2}, concerning properness of the approximating maps, from which Theorem \ref{th:properness} is an immediate consequence. Here we adapt the ideas that appear in \cite{AlarconLopez2012JDG,AlarconCastro2018} to our particular case.
Theorem \ref{th:properness2} is going to be proven using a recursive process
similar to that used to prove our previous results.
The new idea is to construct a sequence of $\Sgot$-immersions, defined on an increasing sequence of domains as before, whose values on compact bands $K_j \backslash \mathring K_{j-1}$ grow like the proper map to be approximated  $F: K \cup M \rightarrow \C^n$.
To achieve it we need the following lemma.

\begin{lemma} \label{lem:propernesslemma}
	Assume that $\Sgot\subset \c^n$, $n\ge 3$, verifies the assumptions of Theorem \ref{th:properness}.
	Let $\Rcal$ be an open Riemann surface. Let $K\Subset L \subset \Rcal$ be smoothly bounded compact \( \Oscr(\Rcal) \)-convex subsets such that $K$ is a strong deformation retract of $L$, let also $\Lambda\subset \mathring K$ be a finite subset, and let $S:=M\cup K\subset \Rcal$ be a Carleman admissible subset 
	{and such that $M\cap (L\setminus \mathring K)$ consists on finitely many pairwise disjoint Jordan arcs connecting $bK$ to $bL$ transversely}.
	Then, given an $\Sgot$-immersion $F\colon M\cup K\to\c^n$, 
	a map $k\colon \Lambda\to\z_+$, and a positive function $\epsilon\colon M\cup K\to\r_+$,
	there exists an $\Sgot$-immersion $\wt F\colon M\cup L\to\c^n$
 such that
	\begin{enumerate}[\rm (a)]
		\item $\| \wt F(p)- F(p)\|<\epsilon(p)$ for any $p\in M\cup K$.
		\smallskip
		\item $\wt F- F$ has a zero {of multiplicity $k(p)$} at each point $p\in\Lambda$.
		\smallskip
		\item $\|\wt F(p)\|_\infty > \frac12\min \{ \|F(q)\|_\infty : q\in bK\cup(M\cap(L\setminus K)) \},\quad p\in \mathring L\setminus K$.
		\smallskip
		\item $\|\wt F(p)\|_\infty > \frac12\min \{ \|F(q)\|_\infty : q\in M\cap bL\},\quad p\in bL.$
	\end{enumerate}
\end{lemma}
\begin{proof}
We may suppose  without loss of generality that $F$ is a holomorphic map on $\Rcal$ by Theorem \ref{th:Carleman}. Note that $L\setminus K$ is a finite pairwise disjoint union of annuli. For simplicity of exposition, we assume that $A:=L\setminus \mathring K$ is connected and hence there is only one annulus to consider; for the general case it is enough to apply the following argument to each component separately. 

Let $\alpha$ denote the component of the boundary of $A$ contained in $bK$ and $\beta$ the other component which is contained in $bL$.
Set
\begin{align*}
\lambda&:=\frac12\min \{ \|F(q)\|_\infty : q\in bK\cup(M\cap(L\setminus K)) \},\\
\tau&:= \frac12\min \{ \|F(q)\|_\infty : q\in M\cap bL\},
\end{align*} 
and $N:=M\cap A$. Note that $\tau \geq \lambda$.
We may assume without loss of generality that \( \lambda, \tau > 0 \).
For if this were not the case we could modify \( K \), \( L \), and \( M \) slightly.
	
To prove the lemma we are going to construct an $\Sgot$-immersion $\wh F\colon \Rcal\to\c^n$ that verifies the following properties: 
	\begin{enumerate}[\rm (i)]
		\item $\| \wh F(p)- F(p)\|<\epsilon(p)$ for any $p\in N\cup K$.
		\smallskip
		\item $\wh F- F$ has a zero of multiplicity $k$ at each point of $\Lambda$.
		\smallskip
		\item $\|\wh F(p)\|_\infty > \lambda
,\quad p\in \mathring L\setminus K$.
		\smallskip
		\item $\|\wh F(p)\|_\infty >\tau
,\quad p\in bL.$
	\end{enumerate}
Then using Lemma \ref{lem:gluing} we will glue $\wh F$ to $F$ at the {\em trouble points} of $bL\cap M$ and get the desired $\Sgot$-immersion.
Suppose that $F=(F_1,\ldots,F_n)$ and $M\cap bK=\{p^1,\ldots,p^r\}$ for some $r\in\n$.
	
Let $\pi_a\colon \c^n\to\c$ denote the $a$-th coordinate projection. Since $F$ satisfies $\|F(p)\|_\infty>\lambda$ for $p\in bK$  there exist an integer $l\ge r$, disjoint subsets $I_1,\ldots,I_n$ of $\z_l=\{0,1,\ldots,l-1\}$ (which denotes the additive cyclic group of integers modulo $l$), and connected subarcs $\{\alpha_j: j\in \z_l \}$ of $bK$, satisfying the following properties: 
\begin{enumerate}[\rm ({a}1)]
	\item $\bigcup_{j\in\z_l} \alpha_j=\alpha$.
		\smallskip
	\item $\alpha_j$ and $\alpha_{j+1}$ have a common endpoint $p_j$, are otherwise disjoint, and the points $p^1,\ldots,p^r\in\{p_1,\ldots,p_l\}$.  In particular, \( \alpha_j \) connects \( p_{j-1} \) to \( p_{j} \).
		\smallskip
	\item $\bigcup_{a=1}^n I_a=\z_l$ and $I_a\cap I_b=\emptyset$ for all $a\neq b\in\{1,\ldots,n\}$.
		\smallskip
	\item If $j\in I_a$ then $|\pi_a(F(p))|>\lambda$ for all $p\in\alpha_j$, $a=1,\ldots,n$.
\end{enumerate}
Note that $I_a$ may be empty  for some $a\in\{1,\ldots,n\}$.
If for some \( j \in \Z_l\) 
there are multiple \( a \in \{ 1, \dots, n\}\) for which $|\pi_{a}(F(p))|>\lambda$ for all \( p \in \alpha_j \), then we simply choose one of these one \( I_a \) for \( j \) to belong to.

Next, for each $j\in\z_l$ consider a smooth Jordan arc $\gamma_j\subset A$ such that:
\begin{itemize}
	\item $\gamma_j$ connects $p_j\in\alpha$ to one point $q_j\in\beta$ and is otherwise disjoint from $bA$.
		\smallskip
	\item If $\gamma_j\cap\alpha=\{p^i\}$ for some $i=1,\ldots,r$ then $\gamma_j\subset N$. Hence $N\subset \bigcup_{j\in\z_l}\gamma_j$.
		\smallskip
	\item The arcs $\gamma_j$ are pairwise disjoint.
\end{itemize}
We may assume that if $\gamma_j\subset N$, then  there is some fixed  \( a \in \{1, \dots, n \}\) such that $|\pi_a(F(p))|>\lambda$ for \( p \in \gamma_j \) and $|\pi_a(F(q_j))|>\tau$. If this were not the case,  then we could take a finite exhaustion of \( L \) starting with \( K \) so that  each pair of consecutive sets in the sequence has this property. Then applying the following reasoning a finite number of times yields the desired result.

Let $S$ be the admissible set given by
\[
	S:=M\cup K\cup (\bigcup_{j\in\z_l}\gamma_j)
\]
Next, we extend $F$ to $S$ such that:
\begin{enumerate}[\rm ({b}1)]
	\item If $j\in I_a$ then $|\pi_a(F(p))|>\lambda$ for all $p\in\gamma_{j-1}\cup\gamma_j$, $a=1,\ldots,n$.
		\smallskip
	\item If $j\in I_a$ then $|\pi_a(F(q_{j-1}))|>\tau$ and $|\pi_a(F(q_j))|>\tau$, $a=1,\ldots,n$.
\end{enumerate}
Recall  we have assumed that if $\gamma_j\subset N$ then $F$  automatically satisfies these properties on \( \gamma_j\).

By  Remark \ref{fattening} and Lemma \ref{lem:gluing}, we may assume that \( F \) is defined on a slightly larger  set of the form \( K' \cup M \), where \( K' \) is a smoothly bounded compact domain.
Then Theorem \ref{th:Carleman} provides an $\Sgot$-immersion $G=(G_1,\ldots,G_n)\colon \Rcal\to\c^n$ with the following properties:
\begin{enumerate}[\rm ({c}1)]
	\item $\|G(p)-F(p)\|<\epsilon(p)$ for all $p\in S$. (Recall we have extended $F$ to $S$.)
		\smallskip
	\item $G-F$ has a zero of multiplicity $k(p)$ at each point $p\in\Lambda$.
		\smallskip
	\item If $j\in I_a$ then $|\pi_a(G(p))|>\lambda$ for all $p\in\gamma_{j-1}\cup\alpha_j\cup\gamma_j$, $a=1,\ldots,n$.
		\smallskip
	\item If $j\in I_a$ then $|\pi_a(G(q_{j-1}))|>\tau$ and $|\pi_a(G(q_{j}))|>\tau$, $a=1,\ldots,n$.
\end{enumerate}
For \( j \in \z_l \) let $\beta_j\subset \beta$ be an arc  connecting $q_{j-1}$ to $q_j$ and not intersecting \( \{ q_1, \dots, q_{l-1} \} \) other than at its endpoints. 
Notice that
\begin{equation}\label{eq:beta}
	\beta=\bigcup_{j\in\z_l}\beta_j.
\end{equation}
Let $D_j\subset A$ be the compact disc bounded by $\alpha_j$, $\beta_j$, $\gamma_j$ and $\gamma_{j-1}$. 
Clearly
\begin{equation}\label{eq:discA}
	A=\bigcup_{j\in\z_l} D_j.
\end{equation}
Set $G_0:=G|_L\colon L\to\c^n$. 
To prove the lemma we are going to construct a sequence of $\Sgot$-immersions $G_m\colon L\to\c^n$ for $m=1,\ldots,n$ that satisfies the following properties:
\begin{enumerate}[\rm ({d}1$_m$)]
	\item $G_m$ approximates $G_{m-1}$ well\footnote{More precisely, \(\displaystyle \| G_m(p) - G_{m-1}(p) \| <  \min_{p \in N \cup K} \frac{\epsilon(p) - \| G(p) - F(p) \|}{n} \) for every \( p \in L \setminus \cup_{j \in I_m}\mathring D_j \).}   on 
$L\setminus (\bigcup_{j\in I_m}\mathring D_j)$.
	\smallskip
	\item $G_m-G_{m-1}$ has a zero of multiplicity $k(p)$ at any point $p\in\Lambda$.
		\smallskip
	\item If $j\in \bigcup_{i=1}^m I_i$ then $||G_m(p)||_\infty>\lambda$ for all $p\in D_j$.
		\smallskip
	\item If $j\in \bigcup_{i=1}^m I_i$ then $||G_m(q)||_\infty>\tau$ for all $q\in \beta_j$.
		\smallskip
	\item If $j\in I_a$ then $|\pi_a(G_m(p))|>\lambda$ for all $p\in\gamma_{j-1}\cup\alpha_j\cup\gamma_j$, $a=1,\ldots,n$.
		\smallskip
	\item If $j\in I_a$ then $|\pi_a(G_m(q))|>\tau$ for $q\in\{q_{j-1},q_j \}$, $a=1,\ldots,n$.
\end{enumerate}

Suppose that we have already constructed such a sequence \( G_1, \dots, G_n \). Then by {\rm (c1)} and {\rm (d1$_{1}$)}--{\rm (d1$_n$)} we may apply Theorem \ref{th:Carleman} 
to \( G_n \) to obtain a map \( \widehat F\colon \Rcal\to\c^n \) that satisfies {\rm (i)} and such that
\begin{equation}\label{eq:jetinter}
	\text{$\wh F-G_n$ has a zero of multiplicity $k$ at each point of $\Lambda$.}
\end{equation}
Condition {\rm (ii)} is then guaranteed by {\rm (d2$_1$)}--{\rm (d2$_n$)}, {\rm (c2)}, and \eqref{eq:jetinter}. 
Condition {\rm (iii)} follows from {\rm (d3$_n$)}, {\rm (a3)}, and \eqref{eq:discA}. Finally condition {\rm (iv)} is implied by {\rm (d4$_n$)}, {\rm (a3)}, and \eqref{eq:beta}. 

To finish the proof let us construct the sequence $\{G_1,\ldots,G_n \}$. We reason by induction. It is clear that $G_0=G|_L$ verifies properties {(\rm d5$_0$)} and {(\rm d6$_0$)} whereas {(\rm d1$_0$)}--{(\rm d4$_0$)} are vacuous. Now, assume that we have constructed $G_{m-1}\colon L\to\c^n$ with $1 \leq m\leq n$ and let us construct $G_{m}\colon L\to\c^n$.

By the continuity of $G_{m-1}$ and the properties {\rm (d5$_{m-1}$)} and {\rm (d6$_{m-1}$)} we have that for any $j\in I_m$ there exists a small compact tubular neighbourhood $\Upsilon_j\subset D_j$ relative to $D_j$ of $\gamma_{j-1}\cup\alpha_j\cup\gamma_j$ such that 
\begin{enumerate}[\rm A)]
	\item[\rm A)] $|\pi_m(G_{m-1}(p))|>\lambda$ for all $p\in \Upsilon_j$.
	\item[\rm B)] $|\pi_m(G_{m-1}(q))|>\tau$ for all $q\in \overline{\beta_j\cap\Upsilon_j}$.
\end{enumerate}
Define
\[
	S_m:=  L\setminus(\bigcup_{j\in I_m}\mathring \Upsilon_j).
\] 
Let $H\colon S_m\to\c^n$ be an $\Sgot$-immersion such that
\begin{enumerate}[\rm C)]
	\item[\rm C)] $H=G_{m-1}$ on $L\setminus(\bigcup_{j\in I_m}D_j)$.
	\item[\rm D)] $\pi_m(H)=\pi_m(G_{m-1})|_{S_m}$.
	\item[\rm E)] $||H(p)||_{\infty}>\tau\geq\lambda$ on $\bigcup_{j\in I_m} (\overline{D_j\setminus \Upsilon_j})$
\end{enumerate}
For instance, one could define $\pi_m(H)=\pi_m(G_{m-1})|_{S_m}$ and for $b\neq m$,
\[
 \pi_{b}(H(z))=\begin{cases}
\pi_{b}(G_{m-1}(z)), & z\in L\setminus(\bigcup_{j\in I_m}D_j).\\
\zeta, & z\in  \bigcup_{j\in I_m} D_j\setminus \Upsilon_j.
\end{cases}
\]
where $\zeta\in\c$ is sufficiently large so that {\rm E)} holds. Notice that $H$ is a holomorphic $\Sgot$-immersion since ${D_j\setminus \Upsilon_j}$ and $L\setminus(\bigcup_{j\in I_m}D_j)$ are separated.

Observe that  $\pi_m(H)$ clearly extends to all of $L$ since $\pi_m(H)=\pi_m(G_{m-1})$. Thus, we may apply \cite[Proposition 5.2]{AlarconCastro2018} to $H$ to obtain a holomorphic $\Sgot$-immersion $G_m\colon L\to\c^n$ such that
\begin{enumerate}[\rm A)]
	\item[\rm F)] $G_m$ approximates $H$ uniformly on $S_m$.
	\item[\rm G)] $G_m- H$ has a zero of multiplicity $k(p)$ at each point $p\in\Lambda$.
	\item[\rm H)] $\pi_m(G_m)=\pi_m(H)=\pi_m(G_{m-1})$ on $L$.
\end{enumerate}
Notice that \cite[Proposition 5.2]{AlarconCastro2018} provides jet-interpolation of fixed order. To ensure property {\rm G)} we apply the result to the integer $\max\{k(p):p\in\Lambda \}\in\z_+$.

We claim that $G_m$ satisfies {\rm (d1$_m$)}--{\rm (d6$_m$)}. Let us check it, {\rm (d1$_m$)} follows from {\rm C)} and {\rm F)}. Properties {\rm C)} and {\rm G)} imply condition {\rm (d2$_m$)}. {\rm (d3$_m$)} is a consequence of {\rm A)} with {\rm H)}, {\rm C)}, {\rm E)}, and {\rm F)}. Similarly, {\rm (d4$_m$)} is guaranteed by {\rm B)} with {\rm H)}, {\rm C)}, {\rm E)}, and {\rm F)}. Finally {\rm (d5$_m$)}  follows from {\rm C)}, {\rm F)}, and {\rm (d5$_{m-1}$}), while {\rm (d6$_m$)} follows from {\rm C)}, {\rm F)}, and {\rm (d6$_{m-1}$}).
\end{proof}

Once we have finished the previous lemma, we proceed with the proof of the main theorem of the section.

\begin{theorem}\label{th:properness2}
	The $\Sgot$-immersion $\wt F\colon \Rcal\to\c^n$, $n\ge 3$, constructed in Theorem \ref{th:Carleman} may be chosen to be proper provided that
	the restricted map $F|_{S \cup \Lambda}$ is proper and $\Sgot$ satisfies the hypotheses of Theorem \ref{th:properness}. 
\end{theorem}

\begin{proof}
Consider a compact exhaustion $\{K_j\}_{j\in\n}$ of $\Rcal$ by \( \Oscr({\Rcal}) \)-convex subsets
as in the proof of Theorem \ref{th:completeness2}.
Since $F|_S$ is proper we assume that $F$ is nonzero on $bK_j\cap M$ and \( \epsilon(p) < 1 \) for all \( p \in M \cup K \).

Set for any $j\in\n$.
 \begin{equation}\label{eq:Sj}
 S_j := bK_{j-1} \cup \big( (\Lambda \cup M \cup K ) \cap (K_j \setminus \mathring K_{j-1})\big).
 \end{equation}
Consider \( F_0:=F\colon M \cup K \cup \Omega \to \mathbb{C}^n \) and $K_0=\emptyset$.  To prove the theorem we will construct a sequence of \( \mathfrak{S} \)-immersions \( F_j\colon M \cup K \cup \Omega\cup K_j \to \mathbb{C}^n \) with the following properties for any $j\in\n$:
\begin{enumerate}[\rm (I$_j$)]	
	\item $\| F_{j}(p)- F_{j-1}(p)\|<1/2^j$ for any $p\in K_{j-1}$.
	\smallskip
	\item  $\| F_{j}(p)- F_{j-1}(p)\|<{\epsilon(p)}/{2^j}$ for any $p\in M\cup K$.
	\smallskip
	\item $F_{j}- F_{j-1}$ has a zero of multiplicity $k$ at each point of $\Lambda$.
	\smallskip
	\item If $F|_\Lambda$ is injective, then $F_j|_{K_j}$ is also injective.
	\smallskip
	\item \( \| F_j(p) \|_\infty > \lambda_j:=\frac{1}{2} \min \{ \| F_{j-1}(q) \|_\infty : q \in S_j \} \) for \( p \in \mathring K_{j} \setminus K_{j-1} \).
	\smallskip
	\item \( \| F_j(p) \|_{\infty} >\tau_j:= \frac{1}{2} \min \{ \| F_{j-1}(q) \|_\infty : q \in M \cap bK_{j} \} \) for \( p \in bK_j \).
\end{enumerate}
Assume for a moment that we have constructed such a sequence. As in the proof of Theorem  \ref{th:Carleman} properties {\rm (I$_{j}$)}--{\rm (IV$_{j}$)} for \( j \in \mathbb{N} \) ensure that there exists a limit map 
\begin{align*}
\widetilde F \colon \mathcal{R} \to \mathbb{C}^n,  \quad \widetilde F = \lim\limits_{j \rightarrow \infty} F_j
\end{align*}
that verifies conclusion of Theorem \ref{th:Carleman}.
On the other hand, it is clear that $\lambda_j$ and $\tau_j$ go to infinity when $j\to\infty$ since $F|_{M\cap K}$ is proper. Therefore, the image by $\wt F$ of a divergent path is also a divergent path. This implies that $\wt F$ is a proper map and would conclude Theorem \ref{th:properness2}.

In order to finish the proof of the theorem, we shall recursively construct the  sequence  \( \{F_j\}_{j \in \N} \).
The starting map is \( F_0 := F \).
Now, assume we are given an \( \mathfrak{S} \)-immersion \( F_{j-1} \colon M \cup K \cup \Omega  \cup K_{j-1}\to \C^n \) for any \( j \ge 1  \) and
let us construct a map \( F_{j} \) satisfying ({\rm I$_{j}$})--({\rm VI$_{j}$}).
First, by shrinking the sets $\Omega_p$ if necessary, we may assume that
\begin{equation}\label{eq:valueonOmega}
	||F_{j-1}(q)||_\infty >\lambda_j, \quad \text{for any $q\in \Omega \cap (K_{j} \setminus \mathring K_{j-1})$}.
\end{equation}
Let \( S \subset K_j\) be a connected $\Oscr(\Rcal)$-convex compact admissible subset obtained by attaching finitely many arcs to 
\[
K_{j-1} \cup \left( (\Omega  \cup K) \cap K_j\right)
\]
so that \( S \) is a  strong deformation retract of \( K_j \) and such that 
\begin{equation}\label{eq:goodintersections2}
\begin{array}{c}
\text{		$M\cap (K_j\setminus \mathring{S})$ consists of finitely many pairwise disjoint Jordan arcs}\\
\text{connecting $bS$ to $bK_j$ and meeting each boundary transversely.}
\end{array}
\end{equation}
Observe that some of the attached arcs
describe the topology of \( K_{j} \setminus K_{j-1} \), if nontrivial, whilst the others connect \( K_{j-1} \) to the connected components of \( K \) inside \(K_{j}\setminus \mathring K_{j-1} \), to the subsets \( \Omega_p \) for $p\in\Lambda \cap (K_j \setminus \mathring K_{j-1})$, and to the arcs of $M$ that do not intersect $K_{j-1}$. 
Note also that we allow the connecting arcs to be contained in  \( M \). See Figure \ref{fig:balloons}.

We may extend \( F_{j-1} \) to an  \( \Sgot \)-immersion on all of 
\( M \cup K \cup \Omega  \cup S \) such that
\begin{equation}\label{eq:valuecurves}
	||F_{j-1}(q)||_\infty>\lambda_j \quad \text{for any $q\in \big(S\cup (M\cap K_j)\big)\setminus \mathring K_{j-1}$,}
\end{equation}
see \eqref{eq:valueonOmega} and the definition of the number $\lambda_j$ in {\rm (V$_j$)}.

Next, we apply Theorem \ref{th:Carleman} (see Remark \ref{rem:avoidtopology}) to the $\Sgot$-immersion
	$F_{j-1}\colon  M \cup K \cup \Omega  \cup S \to\c^n$ 
and a sufficiently small positive continuous function $\epsilon'\colon M \cup K  \cup S \to\r_+$  
to obtain an $\Sgot$-immersion $G\colon \Rcal\to\c^n$ that, by \eqref{eq:valuecurves}, verifies 
\begin{enumerate}[\rm (a)]
	\item \( \| G(p) - F_{j-1}(p)  \|< \epsilon'(p) \) for any \( p \in M \cup K  \cup S  \).
	\item \( G - F_{j-1} \) has a zero of multiplicity \( k \) at each point of \( \Lambda \).
\end{enumerate}
From property {\rm (a)} we have that 
\begin{equation} \label{eq:bKjcondition}
		||G(q)||_\infty>\tau_j \quad \text{for any $q\in bK_j\cap M$}
\end{equation}
and that there exists a small smoothly bounded tubular neighbourhood \( \wt S  \subset K_{j}\) of \( S \)  such that 
\begin{equation} \label{eq:Ktildecondition}
		||G(q)||_\infty>\lambda_j \quad \text{for any \(q\in  (\wt S \setminus \mathring K_{j-1}) \cup \left(M \cap (K_j \setminus \wt S)\right) \)}
\end{equation}
See Figure \ref{fig:balloons}.
\begin{figure}[h]
	\includegraphics[width=11.5cm]{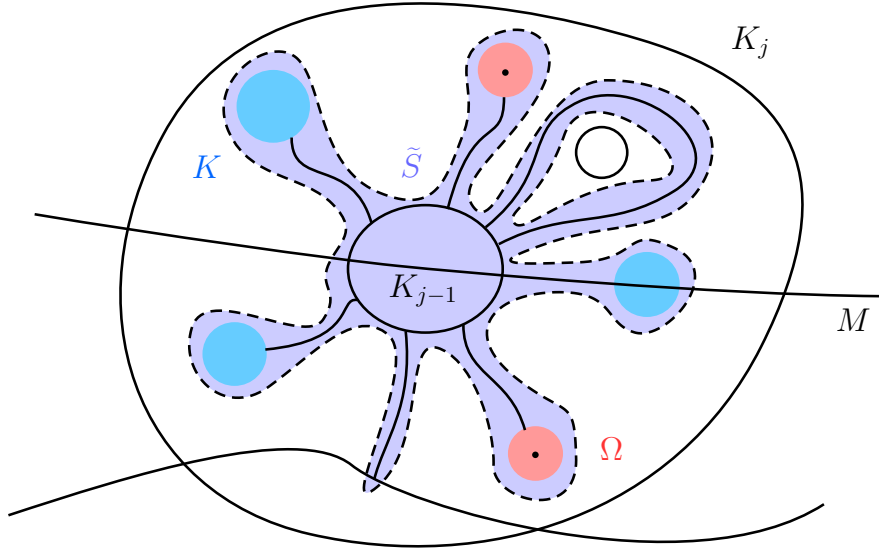}
	\caption{\small{The set $\wt S$ in the case that $K$ has $3$ connected components, $\Lambda$ has $2$ points in $K_j\setminus K_{j-1}$, $M$ has $2$ connected component, one intersects $K_{j-1}$ and the other does not, and $1$ curve is needed for the topology.}}
	\label{fig:balloons}
\end{figure}

Thus, we have that since $S$ is a strong deformation retract of $K_j$ and $\wt S$ is a  tubular neighbourhood of $S$, then  $\wt S$ is a strong deformation retract of $K_j$. Furthermore, by \eqref{eq:goodintersections2}, we may arrange that
$M\cap (K_j\setminus {\wt S}^{{}^\circ})$  also consists of finitely many pairwise disjoint Jordan arcs connecting $b\wt S$ to $bK_j$ and meeting each boundary transversely.
Then we apply Lemma \ref{lem:propernesslemma} to the data
 \[
   \wt S\Subset K_j,\quad \Lambda\cap K_j\subset {\wt S},\quad M, \quad G\colon M\cup \wt S\to\c^n, \quad\text{and}\quad \wt \epsilon\colon M\cup \wt S\to\r_+,
\]
where $\wt\epsilon$ is a positive continuous function such that
\[
	\wt\epsilon(p)<\delta(p)=
	\begin{cases}
	\frac{1}{2^j} & p\in \wt S\setminus (K\cup M),\\
	\min \{ \frac{1}{2^j},\frac{\epsilon(p)}{2^j} \} & p\in (K\cup M)\cap \wt S,\\
	\frac{\epsilon(p)}{2^j} & p\in (K\cup M)\setminus \wt S,
	\end{cases}
\]
to obtain an $\Sgot$-immersion $F_j\colon M\cup K_j\to\c^n$  such that
\begin{enumerate}[\rm (a)]
		\item[\rm (c)] $\| F_j(p)- G(p)\|<\wt \epsilon(p)$ for any $p\in M\cup \wt S$.
		\smallskip
		\item[\rm (d)] $ F_j- G$ has a zero {of multiplicity $k$} at each point of $\Lambda\cap K_j$.
		\smallskip
		\item[\rm (e)] $\| F_j(p)\|_\infty > \frac12\min \{ \|G(q)\|_\infty : q\in b \wt S \cup \big(M \cap (K_j \setminus \wt S)\big)  \},\quad p\in \mathring K_j\setminus \wt S$.
		\smallskip
		\item[\rm (f)] $\| F_j(p)\|_\infty > \frac12\min \{ \|G(q)\|_\infty : q\in M\cap bK_j\},\quad p\in bK_j.$
\end{enumerate}
Note that we may extend \( F_j \) to \( M\cup K\cup \Omega\cup K_j \) by defining  
\begin{align} F_{j}|_{(K\cup\Omega) \setminus K_j}=F_{j-1}  \label{extending}
\end{align}
 since $(K\cup\Omega) \setminus K_j$ is separated from $K_j$.
Moreover, by Theorem \ref{th:Carleman} we may assume also that
\begin{enumerate}[\rm (a)]
\item[\rm (g)] $F_j$ is injective if $G|_\Lambda$ is injective, which is by {\rm (b)}.
\end{enumerate}

We claim that the map $F_j\colon M\cup K\cup \Omega\cup K_j \to\c^n$  verifies the desired conditions {\rm (I$_j$)}--{\rm (VI$_j$)}. Clearly {\rm (I$_j$)}  and {\rm (II$_j$)} follow from {\rm (a)}, {\rm (c)}, and \eqref{extending} since \(  K_{j-1} \subset S \subset \wt S \) and provided that the approximation in {\rm (a)} is good enough.
Condition {\rm (III$_j$)} follows from {\rm (b)}, {\rm (d)}, and \eqref{extending}. Property {\rm (g)} equals condition {\rm (IV$_j$)}
Finally, condition {\rm (V$_j$)} is implied by {\rm (e)}, whereas 
condition   {\rm (VI$_j$)} follows from  {\rm (f)}, providing the approximation in {\rm (a)} is close enough.
\end{proof}

\section{Results for conformal minimal immersions}\label{sec:minimal}

As announced we are going to show how the ideas of Sections \ref{sec:Carleman}, \ref{sec:completeness}, and \ref{sec:properness} may be adapted in order to prove Carleman type approximation results for conformal minimal immersions. 

Recall that the punctured null quadric $\Agot_*=\Agot\setminus\{0\}\subset \c^n$, see \eqref{eq:pnullquadric}, for any $n\ge 3$ directs null curves and is an Oka manifold verifying the additional hypothesis of Theorem \ref{th:easyCarleman}. See \cite[\textsection 2.3]{AlarconCastro2018} for details. Therefore, the results in the previous sections apply and provide complete null curves (and hence conformal minimal immersions with vanishing flux map) that may be chosen proper when the initial one is. However, to get complete and proper conformal minimal immersion with any prescribed flux map some minor modifications have to be made in the proofs. We shall leave the obvious details of the proofs to the interested reader.

First, we are going to say that $X\colon S\to\r^n$, $n\ge 3$, is a {conformal minimal immersion} defined over a Carleman admissible subset $S=M\cup K$ if $X$ is a continuous map such that $X|_K$ is a conformal minimal immersion in the standard case. Notice that, as was the case for $\Sgot$-immersions, we do not require $X$ to be smooth on $M\setminus K$ since we are dealing with $\Cscr^0$ approximation; we could always include a preparatory first step where we approximate  by a smooth map satisfying the extra conditions.

Accordingly, we show the following result.
\begin{theorem}[Carleman theorem with jet interpolation for conformal minimal immersions]\label{th:Carlemanminimal}
	Let $\Rcal$, $S=M\cup K$, $\Lambda$, $\Omega$, $k\colon \Lambda\to\z_+$, and $\epsilon\colon M\cup K\to\r_+$, be as in Theorem \ref{th:Carleman}.
	Given a group morphism $\pgot\colon\Hcal_1(\Rcal;\z)\to\r^n$ and a conformal minimal immersion $X\colon M\cup K\cup\Omega\to\r^n$ such that
	${\mathrm{Flux}}_X=\pgot$ on $K$,
	there exists a complete conformal minimal immersion $\wt X\colon \Rcal\to\r^n$ that verifies the following properties:
	\begin{enumerate}[\rm (i)]
		\item $\| \wt X(p)- X(p)\|<\epsilon(p)$ for any $p\in M\cup K$.
		\item $\wt X$ and $X$ have a contact of order $k(p)$ at any $p\in\Lambda$.
		\item ${\mathrm{Flux}}_{\wt X}=\pgot$ on $\Rcal$.
		\item If $n\ge 5$ and the map $X|_\Lambda$ is injective, then $\wt X$ may be chosen to be injective.
		\item $\wt X$ may be chosen to be proper provided that $X|_{K\cup M\cup \Lambda}$ is proper.
	\end{enumerate}
\end{theorem}
\noindent Notice that Theorem \ref{th:simpleminimal} follows trivially from Theorem \ref{th:Carlemanminimal}. In addition, we point out that condition on $X|_\Lambda$ and $X|_{K\cup M\cup \Lambda}$ in assertions {\rm (iv)} and {\rm (v)} are necessary.
Indeed, in general, one can not choose the conformal minimal immersion $\wt X$ in {\rm (iv)} to be an embedding, i.e. a homeomorphism onto $\wt X(\Rcal)$ endowed with the subspace topology inherited from $\r^n$, since there are injective maps $\Lambda\to\r^n$ that do not extend to a topological embedding. However, since proper injective immersions $\Rcal\to\r^n$ are embeddings, we can choose $\wt X$ in Theorem \ref{th:Carlemanminimal} to be a proper conformal minimal embedding provided that $n\ge 5$ and $X$ verifies both assumptions of {\rm (iv)} and {\rm (v)}.

Next, let us show how the ideas that appear in Section \ref{sec:Carleman}, \ref{sec:completeness}, and \ref{sec:properness} may be used to prove analogues of Theorem \ref{th:Carleman}, Theorem \ref{th:completeness2}, and Theorem \ref{th:properness2} for minimal surfaces.
Our proofs rely on applying Mergelyan type results to approximate in a compact subset and then using Lemma \ref{lem:gluing} to glue with the initial immersion we are given. Therefore, what we need in order to prove Theorem \ref{th:Carlemanminimal} are Mergelyan type results with interpolation for conformal minimal immersions, they can be found in \cite{AlarconCastro2018}, and an analogue to Lemma \ref{lem:gluing} for conformal minimal immersion. 

First, notice that Lemma \ref{lem:gluing} holds for conformal minimal immersions as follows.
\begin{lemma}\label{lem:gluingminimal}
	If we are given $M$, $K\subset L\subset \Rcal$, and $\epsilon\colon M\cup K\to\r_+$ as in Lemma \ref{lem:gluing}, and also conformal minimal immersions $X\colon M\cup K\to\r^n$ and $\wh X\colon\Rcal\to\r^n$ such that
	\[
	||X(p)-\wh X(p) ||<\epsilon(p), \quad p\in (M\cup K)\cap L,
	\]
	there exists a conformal minimal immersion $\wh X\colon M\cup K\cup L\to\r^n$ such that $\wt X|_L=\wh F|_L$ and
	\[
	||X(p)-\wt X(p) ||<\epsilon(p), \quad p\in M\cup K.
	\]
	Obviously, it is verified that $\mathrm{Flux}_{\wt X}=\mathrm{Flux}_{\wh X}$ on $L$.
\end{lemma}
Note that if $L$ is a $\Oscr(\Rcal)$-convex subset then it is enough for proving the theorem that $\wh X$ is defined on $L$, since we may approximate $\wh X$ by Mergelyan Theorem for conformal minimal immersion \cite[Theorem 1.2]{AlarconCastro2018}. The same proof made in Lemma \ref{lem:gluing} provides such a conformal minimal immersion since the punctured null quadric $\Agot_*$ directing minimal surfaces is an Oka manifold.

Using Lemma \ref{lem:gluingminimal} we  can now  give  the proof of Theorem \ref{th:Carlemanminimal}.

\subsection{Proof of Theorem \ref{th:Carlemanminimal}}

Let us show how Theorem \ref{th:Carlemanminimal} is proved. We reason by induction. Consider an exhaustion $\{K_j\}_{j\in\n}$ of $\Rcal$ as in the proof of Theorem \ref{th:Carleman}.
Next, we take a fixed point $p_0\in \mathring K_1$. We would construct a sequence of conformal minimal immersions $X_j\colon M \cup K \cup \Omega \cup K_j\to\r^n$ that verifies the following properties for any $j\in\n$:
\begin{enumerate}[\rm (I$_j$)]	
	\item $\| X_{j}(p)- X_{j-1}(p)\|<1/2^j$ for any $p\in K_{j-1}$.
	\smallskip
	\item  $\| X_{j}(p)- X_{j-1}(p)\|<{\epsilon(p)}/{2^j}$ for any $p\in M\cup K$.
	\smallskip
	\item $X_{j}$ and $X_{j-1}$ have a contact of order $k(p)$ at each point $p\in\Lambda$.
	\smallskip
	\item $\mathrm{Flux}_{X_j}=\pgot$ on $K_j$.
	\smallskip
	\item If $n\ge 5$ and $X|_\Lambda$ is injective, then $X_j$ may be chosen to be injective.
	\smallskip
	\item	${\dist}_{X_j}(b_0,bK_j)>j-1$.
\end{enumerate}
Notice that properties {\rm (I$_j$)}--{\rm (III$_j$)}, and {\rm (VI$_j$)} are similar to the ones in the directed immersions case. Property {\rm (V$_j$)} needs the additional condition $n\ge 5$.

To prove the existence of such a sequence we follow the proof of Theorem \ref{th:completeness2} with small modifications.
Therefore, throughout the reasoning we use the analogues for conformal minimal immersions of the results \cite[Theorem 1.3]{AlarconCastro2018} and Lemma \ref{lem:gluing} that allow us to control the flux map, that is, we use  \cite[Theorem 1.2]{AlarconCastro2018} and Lemma \ref{lem:gluingminimal}.

Observe that, similarly as happen for the directed immersion case, if $n\ge 5$ then \cite[Theorem 1.2]{AlarconCastro2018} provides conformal minimal immersions which are injective. Hence, since Lemma \ref{lem:gluingminimal} involved gluing on arcs then the conformal minimal immersion $X_j$ may be chosen to be injective and then condition {\rm V$_j$} holds.
Thus we prove the existence of the sequence and conclude taking limits when $j\to\infty$. The limit map $\wt X$ is clearly a conformal minimal immersion and satisfies conditions {\rm (i)}--{\rm (iv)} of the theorem.

Let us show how condition {\rm (v)} is deduced. Assume that $X|_{M\cup K\cup \Lambda}$ is proper.
We follow the ideas in the proof of Theorem \ref{th:properness2}, that is, 
we would construct a sequence of conformal minimal immersions \( F_j\colon M \cup K \cup \Omega\cup K_j \to \r^n \) that satisfies conditions {\rm (I$_j$)}--{\rm (V$_j$)} before and also
\begin{enumerate}[\rm (I$_j$)]	
	\item[\rm (VII$_j$)] \( \| X_j(p) \|_\infty > \lambda_j:=\frac{1}{2} \min \{ \| X_{j-1}(q) \|_\infty : q \in S_j \} \) for \( p \in \mathring K_{j} \setminus K_{j-1} \).
	\item[\rm (VIII$_j$)]	 \( \| X_j(p) \|_{\infty} >\tau_j:= \frac{1}{2} \min \{ \| X_{j-1}(q) \|_\infty : q \in M \cap bK_{j} \} \) for \( p \in bK_j \).
\end{enumerate}
Where $S_j$ is the subset defined in \eqref{eq:Sj}.

Once we have constructed the sequence, conditions {\rm (I$_j$)}--{\rm (V$_j$)} give us the existence of a conformal minimal immersion $\wt X:=\lim_{j \to \infty} X_j$ that interpolates at the points of $\Lambda$, has $\pgot$ as flux map, and is injective if $n\ge 5$. On the other hand, conditions {\rm (VII$_j$)} and {\rm (VIII$_j$)} ensures that the conformal minimal immersion $\wt X$ is proper.
Finally, to show the existence of the sequence we need to prove an analogue to Lemma \ref{lem:propernesslemma} for conformal minimal immersions and apply it recursively. Such a lemma is done following the proof of Lemma \ref{lem:propernesslemma} but using the Mergelyan type results proven in \cite[Section 7]{AlarconCastro2018}.

\section{Examples/Applications}
\label{applications}

In this section we give  some preliminary applications of our theorems.
Of particular interest are our solution to the approximate Plateau problem on divergent arcs in \( \R^n \) (Corollary \ref{cor:plat})  and the existence, for every open Riemann surface \( \Rcal \), of a conformal minimal immersion  \( X \colon \mathcal{R} \to \mathbb{R}^n \)  that passes by every conformal minimal immersion \( Y\colon \overline{D} \to \mathbb{R}^n \). (Corollary \ref{cor:oneminimalsurfacetorulethemall}).

Our first corollary is  the following strengthening of Carleman's theorem.

\begin{corollary}
	Given \( f_1, \dots, f_n, \epsilon \in \Cscr{(\R)} \) with \( \epsilon \) strictly positive, there exists a complete null curve \( F=(F_1,\dots,F_n)\colon \C \to \C^n \) such that 
	\begin{align*}
	|F_j(x)-f_j(x) | < \epsilon(x), \quad \text{ for every \( x \in \R \).}
	\end{align*}
\end{corollary}

We will refer to the image of a continuous map from an interval \( I\subset\r \) into \( \R^n \) as a continuous curve.
The following corollary is an immediate consequence of Theorem \ref{th:Carlemanminimal} and asserts that every continuous curve in $\r^n$, $n\ge 3$, is approximately contained in a minimal surface with any conformal structure.
\begin{corollary}
	Given a continuous curve $\gamma\colon I\to C\subset\r^n$, $n\ge 3$, a positive function $\epsilon\colon I\to\r_+$, and an open Riemann surface $\Rcal$, there exist a complete conformal minimal immersion $X\colon \Rcal\to\r^n$ and a continuous curve on $\Rcal$, $\alpha\colon I\to\Rcal$, such that
	\[
		||\gamma(t) - X (\alpha (t))||<\epsilon(t), \quad t\in I.
	\]
\end{corollary}

Recall that the classical Plateau problem solved by Douglas and Rad\'o in  \cite{Douglas1932TAMS,Rado1930AM} consists on finding a minimal surface bounded by a given closed Jordan curve.
The next corollary provides an approximated solution to a certain Plateau problem when the curve is a \emph{divergent arc} in \( \R^n \).

\begin{corollary}\label{cor:plat}
	Let $\Rcal $ be a Riemann surface (without boundary), $D \subset \Rcal$ be a smoothly embedded (closed) disc, and a point in its boundary $p \in b D$. For every continuous divergent curve $\gamma \colon \R \to \R^n$, $n \geq 3,$ and every positive continuous function $\epsilon\colon \mathbb{R} \to \R^n$, there exists a complete conformal minimal immersion $X \colon \Rcal \setminus ( \mathring D \cup \{ p\}) \to \R^n$ and a parametrisation $\alpha \colon \R \to bD \setminus \{ p\} $ such that 
	\[ \| X(\alpha(t))- \gamma(t)\| < \epsilon(t),\quad \text{for every }t \in \R.\]
	Furthermore, if $n \geq 5$, then we can ensure that $X$ is an embedding.
	
\end{corollary}
\begin{proof}
	Let $\alpha\colon \R \to  b D \setminus \{ p\}$ be a parametrisation of the boundary $b D \setminus \{ p\}$.
	Then $\gamma \circ \alpha^{-1}\colon b D \setminus \{ p\} \to \R^n$  is a continuous function from the divergent arc $b D \setminus \{ p\} \subset \Rcal \setminus \{ p\}$ which we know to be a Carleman admissible set by Example \ref{ex:Rembedded}. Therefore using Theorem \ref{th:Carlemanminimal} we may approximate $\gamma \circ \alpha^{-1}$ along $b D \setminus \{ p\}$ by a conformal minimal immersion $X\colon\Rcal \setminus \{ p\} \to \R^n$ better than $\epsilon \circ \alpha^{-1}$. The restriction of $X$ to $\Rcal \setminus (\mathring D \cup \{ p\})$ yields the desired complete conformal minimal immersion. If $n \geq 5$, then by Theorem \ref{th:Carlemanminimal} we may ensure that $X$ is an embedding.
\end{proof}

Let \( M \) be a Riemann surface of finite conformal type with boundary $bM \not = \emptyset $, and let $\rcal$ be an open Riemann surface. 
Following \cite{Lopez2014JGA} we make the following definitions. 
A  conformal minimal immersion $X\colon \Rcal \to \R^n$ is said to pass $Y\colon M \to \R^n$ if there exist proper regions $\{D _j \}_{j \in \mathbb{N}}$ in $\Rcal$ and biholomorphisms $h_{j}\colon M \to D_j$, $j \in \mathbb{N},$ such that $\{X \circ h_j\}_{j \in \mathbb{N}}\rightarrow Y$ in the $\mathscr{C}^0(M)$-topology.
A conformal minimal immersion  $X\colon \Rcal \to \mathbb{R}^n$ is said to be universal if for every compact Riemann surface $M$ with non-empty boundary and any conformal minimal immersion $Y\colon M \to \R^n$, $X$ passes $Y$.

Lopez \cite{Lopez2014JGA} proved that there exist parabolic complete universal minimal surfaces with weak finite total curvature in $\R^3$.
Note that there is no universal minimal surface corresponding to $\C$ since the biholomorphisms  $h_j$ in general do not exist.
If we instead insist on a weaker local version of universality,  then we can prove that   each open Riemann surface $\Rcal$ has a weakly universal minimal surface $X\colon \Rcal \to \R^n$. More precisely, we say that $X\colon\Rcal \to \R^n$ is weakly universal if for every conformal minimal immersion $Y\colon\overline{\mathbb{D}}\to \R^n$, $X$ passes $Y$. Clearly universality implies weak universality, so existence of weakly universal minimal surfaces follows from Lopez \cite{Lopez2014JGA}. 
However, as far as the authors are aware, their existence for an arbitrary Riemann surface $\Rcal$ has not been investigated until now.

\begin{corollary}\label{cor:oneminimalsurfacetorulethemall}
	For each open Riemann surface $\Rcal$ and each natural number $n \geq 3$ there exists a complete weakly universal minimal surface $X\colon \Rcal \to \R^n$.
\end{corollary}
\begin{proof}
	Let $\{K_{j}\}_{j \in \N}$ be a normal exhaustion of $\Rcal$ by $\Oscr(\Rcal)$-convex compact subsets or $\Rcal$. For each $j \in \mathbb{N}$ let $D_j \subset \mathring K_j \setminus K_{j-1}$ be a compact $\Oscr(\Rcal)$-convex subset that is mapped onto the closed unit disc $\overline{\mathbb{D}}$ by a coordinate chart $\varphi_j$ defined on a neighbourhood of $D_j$.
	Define \[h_j = \varphi^{-1}_j|_{\overline{\mathbb{D}}}\colon \overline{\mathbb{D}} \to D_j.\]
	Note that \( \bigcup_{j \in \N} D_j \) is holomorphically convex and has bounded \( E \) hulls.
	
	The space \( \Cscr(\overline \D,\R^n)\) is separable,  meaning it contains a countable dense subset,  and metrisable.
	Therefore, the subset \( \Fcal:=\text{CMI}(\overline{\mathbb{D}},\R^n) \subset \Cscr(\overline\D,\R^n)\) consisting of all conformal minimal surfaces \( \overline \D \to \R^n \) is also separable. For the notion of separability coincides with that of second countability for metrisable spaces, and the latter is hereditary.
	Let \( \{ f_j: j \in \N \}\subset \Fcal\) be a countable dense subset of \( \Fcal \).

	Define
	\[
	f \colon K= {\bigcup\limits_{j \in \N} \overline{D_j} } \to \R^n , \quad
	f(z)=f_j \circ h_{j}^{-1} \text{ for } z \in  D_j,
	\]
	and
	\[
		\varepsilon \colon K \to \R^+,\quad \varepsilon(z)= 1/j \text{ for } z \in D_j.
	\]
	By Theorem \ref{th:Carlemanminimal} we may find a complete conformal minimal immersion  \( X\colon \Rcal \to \R^n\) such that
	\[
	 \| X(z) - f(z) \| < \varepsilon(z), \quad \text{for \( z \in K \).}
	 \]
	
	Suppose \( Y \colon \overline \D \to \R^n \) and \( \epsilon > 0 \) are given.
	By Mergelyan's theorem for conformal minimal immersions we may suppose that \( Y \)  extends to a conformal minimal immersion on \( \R^2 \).
	Let \( B(Y, \epsilon/2) \subset \Fcal \) be the set of all conformal minimal immersions \( g \colon \overline \D \to \R^n \) such that \( \| Y(z) - g(z) \| < \epsilon /2\) for every \( z \in \overline \D \).
	Clearly this set is infinite.
	Moreover, note that there are infinitely many natural numbers \( j \in \N\) such that \( f_j \in B(Y,\epsilon/2) \) by density.
	Let \( j \in \N \) be a natural number  with \( f_j \in B(Y,\epsilon/2)\) which is sufficiently large so that
	\[
	\varepsilon(z)=1/j< \epsilon/2, \quad \text{	for \( z\in  D_j \).}
	\]

	Then,	for \( z \in \overline\D \) we have
	\begin{eqnarray*}
	\| X \circ h_j (z)- Y(z))\|& \leq& \| X\circ h_j (z) - f_j(z))\|  +\| f_j(z)-  Y(z)\|\\
	&=&\| X \circ h_j(z) - f\circ h_j (z)\| + \| f_j(z) - Y(z) \|\\
	& <& \epsilon/2 + \epsilon/2 = \epsilon.
	\end{eqnarray*}
	 It follows that we can construct a sequence $j_1, j_2, \dots$ such that $\{X \circ h_{j_k}\}_{k \in \mathbb{N}}\rightarrow Y$ in the $\mathscr{C}^0(M)$-topology, as required.

\end{proof}

\subsection*{Acknowledgements}
The first author's research is partially supported by the MINECO/FEDER grant no.\,MTM2017-89677-P, Spain.
The second author's research is supported by grant MR-39237 from ARRS, Republic of Slovenia, associated to the research program P1-0291 Analysis and Geometry.


{\bibliographystyle{abbrv} \bibliography{CarlemanApprox}}

\end{document}